\documentclass{amsart}

\usepackage{float}
\usepackage{amsmath}
\usepackage{amsfonts}
\usepackage{amssymb}
\usepackage{graphicx}
\usepackage{hyperref}
\usepackage{mathrsfs}
\usepackage{pb-diagram}
\usepackage{epstopdf}
\usepackage{amsmath,amsfonts,amsthm,enumerate,amscd,latexsym,curves}
\usepackage{bbm}
\usepackage[mathscr]{eucal}
\usepackage{epsfig,epsf,xypic,epic}
\usepackage{caption}
\usepackage{subcaption}
\usepackage{tikz}
\usepackage{geometry}
\usepackage{tikz-cd}
\usepackage{enumitem}

\usetikzlibrary{matrix,arrows,decorations.pathmorphing}

\newtheorem{theorem}{Theorem}[section]

\newtheorem{lemma}[theorem]{Lemma}
\newtheorem{corollary}[theorem]{Corollary}

\newtheorem{definition}[theorem]{Definition}
\newtheorem{assumption}[theorem]{Assumption}

\newtheorem{example}[theorem]{Example}

\newtheorem{proposition}[theorem]{Proposition}

\newtheorem{remark}[theorem]{Remark}

\newcommand{\inner}[1]{\langle #1\rangle}

\newcommand{\bb}[1]{\mathbb{#1}}
\newcommand{\cu}[1]{\mathcal{#1}}
\newcommand{\til}[1]{\widetilde{#1}}
\newcommand{\msc}[1]{\mathscr{#1}}

\newcommand{\mf}[1]{\mathfrak{#1}}

\begin{document}
	
	\title[]{Tropical Lagrangian multi-sections and smoothing of locally free sheaves over degenerate Calabi-Yau surfaces}
	\author[Chan]{Kwokwai Chan}
	\address{Department of Mathematics\\ The Chinese University of Hong Kong\\ Shatin\\ Hong Kong}
	\email{kwchan@math.cuhk.edu.hk}
	\author[Ma]{Ziming Nikolas Ma}
	\address{Department of Mathematics\\ Southern University of Science and Technology\\ Nanshan District \\ Shenzhen\\ China\\}
	\email{mazm@sustech.edu.cn}
	\author[Suen]{Yat-Hin Suen}
	\address{Center for Geometry and Physics\\ Institute for Basic Science (IBS)\\ Pohang 37673\\ Republic of Korea}
	\email{yhsuen@ibs.re.kr}
	\date{\today}

	\begin{abstract}
		We introduce the notion of tropical Lagrangian multi-sections over a $2$-dimensional integral affine manifold $B$ with singularities, and use them to study the reconstruction problem for higher rank locally free sheaves over Calabi-Yau surfaces. To certain tropical Lagrangian multi-sections $\bb{L}$ over $B$, which are explicitly constructed by prescribing local models around the ramification points, we construct locally free sheaves $\cu{E}_0(\bb{L},{\bf{k}}_s)$ over the singular projective scheme $X_0(B,\msc{P},s)$ associated to $B$ equipped with a polyhedral decomposition $\msc{P}$ and a gluing data $s$. We then find combinatorial conditions on such an $\bb{L}$ under which the sheaf $\cu{E}_0(\bb{L},{\bf{k}}_s)$ is simple. This produces explicit examples of smoothable pairs $(X_0(B,\msc{P},s),\cu{E}_0(\bb{L},{\bf{k}}_s))$ in dimension 2.
	\end{abstract}
	
	\maketitle
	
	\tableofcontents

	\section{Introduction}\label{sec:intro}
	
	\subsection{Background}
	The spectacular Gross-Siebert program \cite{GS1, GS2, GS11} is usually referred as an algebro-geometric approach to the famous Strominger-Yau-Zaslow (SYZ) conjecture \cite{SYZ}. It gives an algebro-geometric way to construct mirror pairs. A polarized Calabi-Yau manifold $\check{X}$ near a large volume limit should admit a toric degeneration $\check{p}:\check{\cu{X}}\to \Delta = \text{Spec}(\bb{C}[[t]])$ to a singular Calabi-Yau variety $\check{X}_0:=\check{p}^{-1}(0)$, which appears as a union of toric varieties that intersect along toric strata. By gluing the fans of these toric components, we obtain an integral affine manifold $\check{B}$ with singularities together with a polyhedral decomposition $\check{\msc{P}}$. Then by choosing a strictly convex multi-valued piecewise linear function $\check{\varphi}$ on $\check{B}$, a mirror family $p:\cu{X}\to \Delta$ can be obtained by the following procedure (the fan construction):
	\begin{enumerate}
		\item Take the discrete Legendre transform $(B, \msc{P}, \varphi)$ of $(\check{B}, \check{\msc{P}}, \check{\varphi})$.
		\item Let $X_v$ be the toric variety associated to the fan $\Sigma_v$, defined around a vertex $v\in B$. Glue these toric varieties together along toric divisors, possibly modified with some twisted gluing data $s$, to obtain a projective scheme $X_0(B,\msc{P},s)$.
		\item Smooth out $X_0(B,\msc{P},s)$ to obtain a family $p:\cu{X}\to \Delta$.
	\end{enumerate}
	This is usually referred as the \emph{reconstruction problem} in mirror symmetry.
	
	The most difficult step is (3), namely, to prove that $ X_0(B,\msc{P},s)$ is smoothable. In \cite{GS1}, Gross and Siebert showed that $ X_0(B,\msc{P},s)$ carries the structure of a \emph{log scheme} and it is \emph{log smooth} away from a subset $Z\subset X_0(B,\msc{P},s)$ of codimension at least 2. The subset $Z$ should be thought of as the singular locus of an SYZ fibration (i.e. Lagrangian torus fibration) on the original side and it was the main obstacle in the reconstruction problem. 
	Inspired by Kontsevich-Soibelman's earlier work \cite{KS_HMS_torus_fibration} and Fukaya's program \cite{Fukaya_asymptotic_analysis}, Kontsevich and Soibelman invented the notion of \emph{scattering diagrams} in the innovative work \cite{KS_scattering} and applied it to solve the reconstruction problem in dimension 2 over non-Archimedean fields. This was extensively generalized by Gross and Siebert \cite{GS11}, in which they solved the reconstruction problem over $\mathbb{C}$ in \emph{any} dimension. Roughly speaking, they defined a notion called \emph{structure}, which consists of combinatorial (or tropical) data called \emph{slaps} and \emph{walls}. Applying Kontsevich-Soibelman's scattering diagram technique, they were able to construct a remarkable {\em explicit} order-by-order smoothing of $ X_0(B,\msc{P},s)$. Recently, in \cite{CLM_smoothing} and \cite{Ruddat_smoothing}, it was shown that purely algebraic techniques were enough to prove the existence of smoothing and this can be applied to a more general class of varieties (called toroidal crossings varieties).
	
	In view of Kontsevich's homological mirror symmetry (HMS) conjecture \cite{HMS}, it is natural to ask if one can reconstruct coherent sheaves from combinatorial or tropical data as well, where the latter should arise as tropical limits of Lagrangian submanifolds on the mirror side. In the rank one case, this was essentially accomplished by the series of works \cite{GHK11, GHS16, Gross-Siebert16}, in which the Gross-Siebert program was extended and applied to reconstruct (generalized) theta functions, or sections of ample line bundles, on Calabi-Yau varieties, proving a strong form of Tyurin's conjecture \cite{Tyurin}. This paper represents an initial attempt to tackle the reconstruction problem in the higher rank case. We will restrict our attention to the dimension 2, as in \cite{GHK11}. 

	\subsection{Main results}
	Recall that the third author of this paper demonstrated in \cite{Suen_TP2} that the tangent bundle $T_{\bb{P}^2}$ of the complex projective plane $\bb{P}^2$ can be reconstructed from some tropical data on the fan of $\bb{P}^2$, which he called \emph{a tropical Lagrangian multi-section}. The definition there, however, is not general enough to cover more interesting cases which could arise in mirror symmetry. In Section \ref{sec:trop_Lag}, we give a more general definition of tropical Lagrangian multi-sections over a $2$-dimensional integral affine manifold $B$ with singularities equipped with a polyhedral decomposition $\msc{P}$.\footnote{We will later assume that $(B, \msc{P})$ positive and simple, which will put constraints on the singularities; see Section \ref{sec:GS}.}
	We expect such an object to arise as a certain kind of tropicalization of Lagrangian multi-sections in an SYZ fibration of the mirror.
	Roughly speaking, it is a quadruple $\bb{L}:=(L,\pi,\msc{P}_{\pi},\varphi')$ consisting of a topological (possibly branched) covering map $\pi:L\to B$, a polyhedral decomposition $\msc{P}_{\pi}$ on $L$ respecting $\msc{P}$ and a multi-valued piecewise linear function $\varphi'$ on $L$. A key difference from the usual polyhedral decomposition is that we require the ramification locus $S' \subset L$ of $\pi:L\to B$ to be contained in the codimension 2 strata of $(L,\msc{P}_{\pi})$. In particular, the pullback affine structure on $L$ is also singular along the ramification locus $S'$ (see Definition \ref{def:trop_lag}).
	\begin{remark}
		Our definition of tropical Lagrangian multi-sections recovers the cone complex associated to a toric vector bundle constructed by Payne \cite{branched_cover_fan}, at least when the cone complex is a smooth manifold.
	\end{remark}
	To construct examples of tropical Lagrangian multi-sections, we need good local models. In Section \ref{sec:local_model}, we describe some explicit local models of $\varphi'$ around the ramification points. To explain these, recall that Payne \cite{branched_cover_fan} used an equivariant structure of $T_{\bb{P}^2}$ on each affine chart to define a piecewise linear function $\varphi_{2,1}$ on a 2-fold covering of $|\Sigma_{\bb{P}^2}|\cong N_{\bb{R}}$ that takes the form
	$$\varphi_{2,1}:=\begin{cases}
	-\xi_1 &\text{ on }\sigma_0^+,\\
	-\xi_2 &\text{ on }\sigma_0^-\\
	\xi_1 &\text{ on }\sigma_1^+,\\
	\xi_1-\xi_2 &\text{ on }\sigma_1^-,\\
	-\xi_1+\xi_2 &\text{ on }\sigma_2^+,\\
	\xi_2 &\text{ on }\sigma_2^-;
	\end{cases}$$
	here $\sigma_i^{\pm}$ are two copies of the cone $\sigma_i$ and $(\xi_1,\xi_2)$ are affine coordinates on $N_{\bb{R}}\cong\bb{R}^2$.
	Our local models are simply given by varying the coefficients of the function $\varphi_{2,1}$ as shown in Figure \ref{fig:L}; here $m,n$ are integers with $m\neq n$ and we call the resulting function $\varphi_{m,n}$.
	We say that a tropical Lagrangian multi-section is \emph{of class $\mathscr{S}$} (denoted as $\bb{L} \in \mathscr{S}$) if the multi-valued piecewise linear function $\varphi'$ is locally modeled by $\varphi_{m,n}$ for some $m,n$ around each ramification point of $\pi$; if the same $m,n$ are used at each ramification point, we obtain a special subclass called $\msc{S}_{m,n}\subset\msc{S}$.
	\begin{figure}[ht]
		\includegraphics[width=120mm]{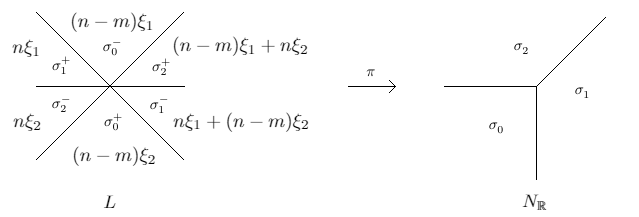}
		\caption{The tropical Lagrangian multi-section $\bb{L}$}
		\label{fig:L}
	\end{figure}
	
	In Section \ref{sec:E_0}, we construct a locally free sheaf $\cu{E}_0(\bb{L},{\bf{k}}_s)$ over the singular variety $X_0(B,\msc{P},s)$ from a tropical Lagrangian multi-section $\bb{L}$ of class $\mathscr{S}$. This is already nontrivial because there are various obstructions to the gluing process. To see that, we fix a vertex $v\in B$ which corresponds to a maximal toric stratum $X_v$. The function $\varphi'$ defines an toric line bundle on the affine chart $U_i:=\text{Spec}(\bb{C}[K_v(\sigma_i)^{\vee}\cap M])\subset X_v$ for each $\sigma_i\in\msc{P}_{max}$ which contains $v$. Taking direct sum produces a rank $r$ toric vector bundle over $U_i$. If the vertex $v\in B$ is not a branch point of $\pi: L \to B$, these local pieces glue to give a rank $r$ bundle over $X_v$ which splits. But when $v$ is a branch point of $\pi$, there are two line bundles $\cu{L}_i^+,\cu{L}_i^-$ over $U_i$ which \emph{cannot} be glued due to nontrivial monodromy around the ramification point $v' \in \pi^{-1}(v)$. In such a case, we follow \cite{Suen_TP2} (which was in turn motivated by Fukaya's proposal for reconstructing bundles in \cite{Fukaya_asymptotic_analysis}) and try to glue $\cu{L}_i^+\oplus\cu{L}_i^-$'s equivariantly to obtain a set of naive transition functions:
	$$\tau^{sf}_{10}:=\begin{pmatrix}
	\frac{a_0}{(w_0^1)^m} & 0
	\\0 & \frac{b_0}{(w_0^1)^n}
	\end{pmatrix},
	\tau^{sf}_{21}:=\begin{pmatrix}
	\frac{b_1}{(w_1^2)^n} & 0
	\\0 & \frac{a_1}{(w_1^2)^m}
	\end{pmatrix},
	\tau^{sf}_{02}:=\begin{pmatrix}
	0 & \frac{b_2}{(w_2^0)^n}
	\\\frac{a_2}{(w_2^0)^m} & 0
	\end{pmatrix}.$$
	The problem is that these do \emph{not} satisfy the cocycle condition. Thus we have to modify $\tau_{ij}^{sf}$ by multiplication by an invertible factor $\Theta_{ij}$ (the wall-crossing factor), namely, $\tau_{ij}:=\tau_{ij}^{sf}\Theta_{ij}$. We then obtain the following
	\begin{proposition}[=Proposition \ref{prop:monodromy_free}]
		If we impose the condition that $\prod_ia_ib_i=-1$, then
		$$\tau_{02}\tau_{21}\tau_{10}=I.$$
		Moreover, the equivariant structure defined by \eqref{eqn:sf_equ_str} can be extended.
	\end{proposition}
	From this we obtain a rank 2 toric vector bundle $E_{m,n}$ on $\bb{P}^2$, and we denote the corresponding rank 2 bundle on the toric component $ X_v\cong\bb{P}^2$ by $\cu{E}(v')$. Taking direct sum with the $r-2$ line bundles $\cu{L}(v^{(\alpha)})$, $\alpha=1,\dots,r-2$, we obtain a rank $r$ bundle $\cu{E}(v)$ on $ X_v$ even when $v$ is a branch point of $\pi$. To glue the bundles $\{\cu{E}(v)\}$ together, a key observation is that the factors $\Theta_{ij}$ act \emph{trivially} on the boundary divisors.\footnote{This is because each irreducible component of the boundary divisor is isomorphic to $\bb{P}^1$ and hence any bundle splits into a direct sum of two line bundles; one should not expect such a nice property for $\dim(B)\geq 3$.} However, there is further obstruction to gluing these bundles together. This obstruction, which is given by a cohomology class $o_{\bb{L}}([s])\in H^2(L,\bb{C}^{\times})$, is analogous to that in \cite[Theorem 2.34]{GS1} that arises from gluing of the projective scheme $X_0(B,\msc{P},s)$.
	
	\begin{theorem}[=Theorem \ref{thm:existence}]\label{thm:existence_intro}
		If $o_{\bb{L}}([s])=1$, then there exists a rank $r$ locally free sheaf $\cu{E}_0(\bb{L},{\bf{k}}_s)$ over $X_0(B,\msc{P},s)$.
	\end{theorem}
	
	We can now proceed to study smoothability of the pair $(X_0(B,\msc{P},s),\cu{E}_0(\bb{L},{\bf{k}}_s))$ (see Section \ref{sec:simple_smooth}). We will assume that the polyhedral decomposition $\msc{P}$ is positive and simple, as in \cite{GS1,GS2,GS11}, as well as \emph{elementary}, meaning that every cell in $\msc{P}$ is an elementary simplex.\footnote{In dimension 2, every polyhedral decomposition can actually be subdivided into elementary simplices (or equivalently, standard simplices).} From the Gross-Siebert program, we already know that $ X_0(B,\msc{P},s)$ can be smoothed out to give a formal polarized family $p:\cu{X}\to \Delta := \text{Spec}(\bb{C}[[t]])$ of Calabi-Yau surfaces.
	We will focus on the sheaves $\cu{E}_0(\bb{L},{\bf{k}}_s)$ which correspond to the tropical Lagrangian multi-sections $\bb{L}\in \mathscr{S}_{n+1,n}$ (the $m \geq n+2$ cases are actually much easier).

	To prove smoothability of the pair $( X_0(B,\msc{P},s),\cu{E}_0(\bb{L},{\bf{k}}_s))$ for $\bb{L}\in \mathscr{S}_{n+1,n}$, our strategy is to apply a result in a previous work of the first and second authors \cite[Corollary 4.6]{CM_pair}, for which we need the condition that $H^2(X_0(B,\msc{P},s),\cu{E}nd_0(\cu{E}_0(\bb{L},{\bf{k}}_s))) = 0$. In general, it is not easy to deal with higher cohomologies. Exploiting the fact that $X_0(B,\msc{P},s)$ is a Calabi-Yau surface and Serre duality, we are reduced to showing that $H^0(X_0(B,\msc{P},s),\cu{E}nd_0(\cu{E}_0(\bb{L},{\bf{k}}_s))) = 0$, or equivalently, that the locally free sheaf $\cu{E}_0(\bb{L},{\bf{k}}_s)$ is \emph{simple}.
	
	In the rank $2$ case, we are able to find a clean and checkable combinatorial condition on the tropical Lagrangian multi-section $\bb{L}$ which is equivalent to simplicity of $\cu{E}_0(\bb{L},{\bf{k}}_s)$ (see Section \ref{sec:rk2}).
	In order to state our main result, we consider the embedded graph $G(\msc{P}) \subset B$ given by the union of all 1-cells in the polyhedral decomposition $\msc{P}$, and then let 
	$$G_0(\bb{L}) \subset G(\msc{P})$$
	be the subgraph obtained by removing all the branch vertices and the adjacent edges.
	We call a 1-cycle $\gamma\subset G_0(\bb{L})$ which bounds a 2-cell in $\msc{P}$ a \emph{minimal cycle} (see Definition \ref{def:minimal cycle}).
	The key observation is that a nontrivial endomorphism of $\cu{E}_0(\bb{L},{\bf{k}}_s)$ gives rise to a minimal cycle in $G_0(\bb{L})$, and vice versa:
	\begin{theorem}[=Theorem \ref{thm:simple}]\label{thm:main intro}
		Let $\bb{L}\in \mathscr{S}_{n+1,n}$. Then the locally free sheaf $\cu{E}_0(\bb{L},{\bf{k}}_s)$ is simple if and only if $G_0(\bb{L})$ has no minimal cycles.
	\end{theorem}
	
	Because of this result, we say that a tropical Lagrangian multi-section $\bb{L}\in \mathscr{S}_{n+1,n}$ is \emph{simple} if $G_0(\bb{L})$ does not contain any minimal cycles (see Definition \ref{def:simple}).
	
	
	\begin{corollary}[=Corollary \ref{cor:smoothable}]\label{cor:corollary1_intro}
		If $\bb{L}\in\msc{S}_{n+1,n}$ is simple, then the pair $( X_0(B,\msc{P},s),\cu{E}_0(\bb{L},{\bf{k}}_s))$ is smoothable.
	\end{corollary}
	
	
	The higher rank case turns out to be much more subtle, and we can only obtain some partial results (see Section \ref{sec:higher rank}).
	For $r \geq 2$, we consider a tropical Lagrangian $r$-fold multi-section $\bb{L}$ over $(B,\msc{P})$ which is totally ramified, i.e. locally modeled by the $r$-fold map $z\mapsto z^r$ on $\bb{C}$ around each branch point $v \in S$ (here $S := \pi(S') \subset B$ denotes the branch locus of $\pi: L \to B$), and satisfies certain slope conditions (see Definition \ref{def:class_C}). This defines another special class of tropical Lagrangian multi-sections, which we call $\cu{C}$. The major problem is that we do not know whether there exist toric vector bundles over the toric component $X_v$ whose cone complexes give exactly these local models (unlike the rank 2 case where we have the bundles $E_{m,n}$). So we make the following:
	\begin{assumption}\label{assumption}
		For each branch point $v \in S$, there exists a rank $r$ toric vector bundle $\cu{E}(v)$ over $X_v$ whose cone complex satisfies the slope conditions in Definition \ref{def:class_C}
	\end{assumption}
	The bundle $\cu{E}(v)$ associated to $v\in S$, if exists, is guaranteed to be simple.
	Also, an analogue of Theorem \ref{thm:existence_intro} holds, namely, the collection $\{\cu{E}(v)\}_{v \in \msc{P}(0)}$ can be glued to produce a locally free sheaf $\cu{E}_0(\bb{L},{\bf{k}}_s)$ over $X_0(B,\msc{P},s)$ when the obstruction vanishes. This puts us in a situation analogous to the rank 2 case.
	
	
	To understand simplicity and smoothability of $\cu{E}_0(\bb{L},{\bf{k}}_s)$, we define a graph analogous to $G_0(\bb{L})$; however, unlike the rank 2 case, such a graph lives in the fiber product
	$$P(\bb{L}):=L\times_BL,$$
	which should be viewed as a certain (fiberwise) path space of $\bb{L}$, instead of $B$.
	There is a natural polyhedral decomposition on $P(\bb{L})$, given by $\til{\msc{P}}:=\msc{P}'\times_{\msc{P}}\msc{P}'$.
	The union of all 1-cells in $\til{\msc{P}}$ gives a graph $\til{G}(\bb{L}) \subset P(\bb{L})$, and we can define the subgraph $\til{G}_0(\bb{L}) \subset \til{G}(\bb{L})$ by suitable removing some edges (see Definition \ref{def:subgraph higher rank} for the detailed definition). We remark that, when $r=2$, the graph $\til{G}_0(\bb{L})$ can be identified with $G_0(\bb{L})$ via the projection $\pi:P(\bb{L})\to B$.
	Then we have the following weakened analogue of Theorem \ref{thm:main intro}:
	\begin{theorem}[=Theorem \ref{thm:general_simple}]
		Let $\bb{L}\in\cu{C}$ such that Assumption \ref{assumption} holds.
		Suppose that $\til{G}_0(\bb{L})$ has no minimal cycles (1-cycle which bound a 2-cell in $\til{\msc{P}}$), 
		and that the line bundle $\cu{L}(\til{v})$ associated to any $\til{v}\in\til{G}_0(\bb{L})$ admits a section $s_{\til{v}}\in H^0(X_v,\cu{L}(\til{v}))$ such that $s_{\til{v}}(p_{\til{v}})\neq 0$ for some torus fixed point $p_{\til{v}}\in X_v$.
		Then $\cu{E}_0(\bb{L},{\bf{k}}_s)$ is simple and hence the pair $(X_0(B,\msc{P},s),\cu{E}_0(\bb{L},{\bf{k}}_s))$ is smoothable.
	\end{theorem}
	For explicit examples of smoothable pairs $(X_0(B,\msc{P},s),\cu{E}_0(\bb{L},{\bf{k}}_s))$ obtained using the above results, see Examples \ref{eg:1}, \ref{eg:2} and \ref{eg:3}.
	
	\subsection{Further perspectives}
	
	We end this introduction by two remarks. 
	\subsubsection{}
	First of all, associated to the locally free sheaf $\cu{E}_0(\bb{L},{\bf{k}}_s)$ is a constructible sheaf $\cu{F}$ on $P(\bb{L}) = L\times_{B}L$ defined as follows: Let $\pi_B:P(\bb{L})\to B$ be the natural projection map. We define
	$$\til{\msc{P}}:=\{\sigma_1'\times_{\pi}\sigma_2' \mid \sigma_1',\sigma_2'\in\msc{P}_{\pi}\text{ such that }\pi(\sigma_1')=\pi(\sigma_2')\}.$$
	Let $\{v\}\in\msc{P}$ be a vertex and $\til{v} = (v_1',v_2') \in \pi_B^{-1}(v)$. Each $v_i'$ gives a vector bundle $\cu{E}(v_i')$ on $ X_v$, which is either a line bundle or a rank 2 bundle. For $\til{\sigma}\in\til{\msc{P}}$ with $\til{v}\in\til{\sigma}$, we define
	$$\cu{F}(\til{\sigma}):=H^0( X_{\sigma},(\cu{E}(v_1')^*\otimes\cu{E}(v_2'))|_{ X_{\sigma}}),$$
	where $\sigma:=\pi_B(\til{\sigma})$.
	For $\til{\tau}\subset\til{\sigma}$, the generalization map
	$g_{\til{\tau}\til{\sigma}}:\cu{F}(\til{\tau})\to \cu{F}(\til{\sigma})$
	is induced by the inclusion $\iota_{\sigma\tau}: X_{\sigma}\hookrightarrow X_{\tau}$. Clearly,
	$g_{\til{\rho}\til{\sigma}}=g_{\til{\tau}\til{\sigma}}\circ g_{\til{\rho}\til{\tau}}$, whenever $\til{\rho}\subset\til{\tau}\subset\til{\sigma}$. Hence the data $\left(\{\cu{F}(\til{\sigma})\},\{g_{\til{\tau}\til{\sigma}}\}\right)$ defines a constructible sheaf $\cu{F}$ on $P(\bb{L})$ (see e.g. \cite{Kapranov-Schechtman16}). 
	
	Let
	$P_0(\bb{L}):=P(\bb{L})\backslash\Delta_L$,
	where $\Delta_L$ denotes the diagonal. By restricting $\cu{F}$ to $P_0(\bb{L})$, we get a sheaf $\cu{F}_0$ on $P_0(\bb{L})$. By construction, we have canonical identifications
	\begin{equation}\label{eqn:ccc}
	\begin{split}
	H^0( X_0(B,\msc{P},s),\cu{E}nd(\cu{E}_0(\bb{L},{\bf{k}}_s))) & \cong H^0(P(\bb{L}),\cu{F}),\\
	H^0( X_0(B,\msc{P},s),\cu{E}nd_0(\cu{E}_0(\bb{L},{\bf{k}}_s))) & \cong  H^0(P_0(\bb{L}),\cu{F}_0).
	\end{split}
	\end{equation}
	The \emph{coherent-constructible correspondence} for toric varieties was established by Fang-Liu-Treumann-Zaslow in  \cite{FLTZ11a, FLTZ11c} and applied to prove (a version of) homological mirror symmetry (HMS) for toric varieties in \cite{FLTZ11b, FLTZ12}. As mentioned above, one should think of the fiber product $P(\bb{L})$ as a certain fiberwise path space of $\bb{L}$.
	More precisely,	for a non-singular SYZ fibration $p:\check{X}\to B$ and an honest Lagrangian multi-section $\bb{L}\subset \check{X}$, one can talk about the fiberwise geodesic path space as in \cite{KS_HMS_torus_fibration, Ma_thesis}, and then a point $(x_1',x_2')\in P(\bb{L})$ can be regarded as the end points of an affine geodesic from $\bb{L}$ to itself.
	The identifications \eqref{eqn:ccc} suggest that if one consider higher rank sheaves $\cu{E}$ on $X$, the self-Hom space of $\cu{E}$ should be computed by a certain (possibly derived) constructible sheaf on the ``path space'' $P(\bb{L})$. We leave this for future research.
	
	\subsubsection{}
	On the other hand, in order to put our results into the context of the HMS conjecture, it is best to apply the framework laid out in the very recent work \cite{gammage2021homological}. According to Seidel \cite{Seidel-ICM, Seidel-K3}, a large volume limit of the mirror $\check{X}$ should be constructed by removing a normal crossing divisor $D$ which represents the K\"ahler class of $\check{X}$, giving rise to a Weinstein manifold $\check{Y}$. This produces a mirror pair $X_0$ and $\check{Y}$ at the large complex structure/volume limits. The HMS conjecture for this pair is much simpler, due to the fact that the Fukaya category of $\check{Y}$ is quasi-equivalent to a category of sheaves with microlocal supports on its Lagrangian skeleton $\Lambda$, as conjectured by Kontsevich. Much has been done along this direction; readers are referred to e.g. \cite{FLTZ11b, ganatra2020covariantly, gammage2021homological}.  
	
	In \cite{gammage2021homological}, a Lagrangian skeleton $\Lambda(\Phi) \subset \check{Y}(\Phi)$ is constructed by gluing the local models in \cite{FLTZ11b} according to a combinatorial structure $\Phi$ called \textit{fanifold}, which can be extracted from $(B,\msc{P})$ (here we assume that the gluing data $s$ is trivial). Furthermore, they proved that there is a quasi-equivalence
	$$
	\text{DCoh}(X_0) \cong \text{Fuk}(\check{Y}(\Phi),\partial_{\infty}\Lambda(\Phi)),
	$$
	where $\text{DCoh}(X_0)$ is the dg category of coherent sheaves on $X_0$ and $\text{Fuk}(\check{Y}(\Phi),\partial_{\infty}\Lambda(\Phi))$ is the partially wrapped Fukaya category on $(\check{Y}(\Phi),\partial_{\infty}\Lambda(\Phi))$.
	We believe that our locally free sheaf $\cu{E}_0(\bb{L},{\bf{k}}_s)$, as an object in $\text{DCoh}(X_0)$, corresponds to a compact connected immersed exact Lagrangian $\mathbb{L}$ in $\check{Y}(\Phi)$. Moreover, it was conjectured in \cite{gammage2021homological} that there should be a fibration $\tilde{\pi} : \check{Y}(\Phi) \rightarrow B$ serving as an SYZ fibration in the large volume limit. If so, the Lagrangian $\mathbb{L}$ would be a Lagrangian multi-section in the fibration $\tilde{\pi}$. 
	
	The simplicity assumption in Definition \ref{def:simple} (in the rank 2 case) corresponds to the isomorphism $HF^0(\bb{L},\bb{L}) \cong \mathbb{C}$ under HMS. In particular, it should be satisfied when $\mathbb{L}$ is connected and embedded in $\check{Y}$. From \cite{Seidel-ICM, Seidel-K3, gammage2021homological}, we expect that the Fukaya category $\text{Fuk}(\check{X},D)$ can be obtained as a deformation of $\text{Fuk}(\check{Y},\partial_{\infty}\Lambda(\Phi))$ by corrections coming from holomorphic disks which intersect $D$. This deformation is mirror to the deformation of $X_0$ that yields the mirror family $p:\cu{X} \rightarrow S$. In this picture, when the Lagrangian $\mathbb{L}$ is connected and embedded, we believe that there are no holomorphic disks in $\check{X}$ bounded by $\mathbb{L}$. This would imply that $\mathbb{L}$ deforms as an object in the category $\text{Fuk}(\check{X},D)$. Therefore, the locally free sheaf $\cu{E}_0(\bb{L},{\bf{k}}_s)$ should indeed be mirror to the Lagrangian multi-section $\mathbb{L}$.
	
	\subsection*{Acknowledgment}
	We are indebted to the anonymous referees for very constructive comments and suggestions, which have helped to greatly improve the exposition. The third author is grateful to Yamamoto Yuto for useful and joyful discussions. We would also like to thank Yong-Geun Oh and Cheol-Hyun Cho for their interest in this work. 
	
	The work of K. Chan was supported by grants of the Hong Kong Research Grants Council (Project No. CUHK14302617 \& CUHK14303019) and direct grants from CUHK.	The work of Z. N. Ma was supported by the Institute of Mathematical Sciences (IMS) and the Department of Mathematics at The Chinese University of Hong Kong. The work of Y.-H. Suen was supported by IBS-R003-D1.
	
	\section{The Gross-Siebert program}\label{sec:GS}
	
	In this section, we review some machinery in the Gross-Siebert program, mainly following \cite{GS1}.
	
	\subsection{Affine manifolds with singularities and their polyhedral decompositions}
	
	\begin{definition}[\cite{GS1}, Definition 1.15]\label{def:affine manifold with sing}
		An \emph{integral affine manifold with singularities} is a topological manifold $B$ together with a closed subset $\Gamma\subset B$, which is a finite union of locally closed submanifolds of codimension at least 2, such that $B_0:=B\backslash\Gamma$ is an integral affine manifold (meaning that the transition functions are integral affine).
		
		Let $B$ be an integral affine manifold with singularities and $U\subset B$ be an open subset. A continuous function $f:U\to\bb{R}$ is call \emph{integral affine} if $f|_{U\cap B_0}:U\cap B_0\to\bb{R}$ is an integral affine function. The sheaf of integral affine functions on $B$ is denoted by $\cu{A}ff(B,\bb{Z})$.
	\end{definition}
	
	Fix a rank $n$ free abelian group $N\cong\bb{Z}^n$ and let $N_{\bb{R}}:= N \otimes_{\bb{Z}}\bb{R}$.
	
	\begin{definition}[\cite{GS1}, Definition 1.21]
		A \emph{polyhedral decomposition of a closed subset $R\subset N_{\bb{R}}$} is a locally finite covering $\mathscr{P}$ of $R$ by closed convex polytopes (called \emph{cells}) with the property that
		\begin{enumerate}
			\item If $\sigma\in\mathscr{P}$ and $\tau\subset\sigma$ is a face, then $\tau\in\mathscr{P}$.
			\item If $\sigma,\sigma'\in\mathscr{P}$, then $\sigma\cap\sigma'$ is a common face of $\sigma,\sigma'$.  
		\end{enumerate}
		We say $\mathscr{P}$ is \emph{integral} if all vertices (0-dimensional cells) are contained in $N$.
	\end{definition}
	
	For a polyhedral decomposition $\mathscr{P}$ and a cell $\sigma\in\mathscr{P}$, we denote the relative interior of $\sigma$ by
	$$\text{Int}(\sigma):=\sigma\Big\backslash\bigcup_{\tau\in\mathscr{P},\tau\subsetneq\sigma}\tau.$$
	
	\begin{definition}[\cite{GS1}, Definition 1.22]\label{def:poly_decomp}
		Let $B$ be an integral affine manifold with singularities. A \emph{polyhedral decomposition of $B$} is a collection $\mathscr{P}$ of closed subsets of $B$ (called \emph{cells}) covering $B$ which satisfies the following properties. If $\{v\}\in\mathscr{P}$ for some $v\in B$, then $v\notin\Gamma$, and there exist a polyhedral decomposition $\mathscr{P}_v$ of a closed subset $R_v\subset T_vB\cong\Lambda_v\otimes\bb{R}$, which is the closure of an open neighborhood of $0\in T_vB$, and a continuous map $\exp_v:R_v\to B$ with $\exp_v(0)=v$ satisfying the following conditions:
		\begin{enumerate}
			\item $\exp_v$ is a local homeomorphism onto its image, is injective on $\text{Int}(\tau)$ for all $\tau\in\mathscr{P}_v$, and is an integral affine map in some neighborhood of $0\in R_v$.
			\item For every top dimensional cell $\til{\sigma}\in\mathscr{P}_v$, $\exp_v(\text{Int}(\til{\sigma}))\cap\Gamma=\emptyset$ and the restriction of $\exp_v$ to $\text{Int}(\til{\sigma})$ is an integral affine map. Furthermore, $\exp_v(\til{\tau})\in\mathscr{P}$ for all $\til{\tau}\in\mathscr{P}_v$.
			\item $\sigma\in\mathscr{P}$ and $v\in\sigma$ if and only if $\sigma=\exp_v(\til{\sigma})$ for some $\til{\sigma}\in\mathscr{P}_v$ with $0\in\til{\sigma}$.
			\item Every $\sigma\in\mathscr{P}$ contains a point $v$ with $\{v\}\in\mathscr{P}$.
		\end{enumerate}
		We say the polyhedral decomposition is \emph{toric} if it satisfies the additional condition:
		\begin{enumerate}[resume]
			\item For each $\sigma\in\msc{P}$, there is a neighborhood $U_{\sigma}\subset B$ of $\text{Int}(\sigma)$ and an integral affine submersion $S_{\sigma}:U_{\sigma}\to N_{\bb{R}}'$, where $N'$ is a lattice of rank ${\dim(B)-\dim(\sigma)}$ and $S_{\sigma}(\sigma\cap U_{\sigma})=\{0\}$.
		\end{enumerate}
		A polyhedral decomposition of $B$ is called \emph{integral} if all vertices are integral points of $B$.
	\end{definition}

	The $k$-dimensional strata of $(B,\mathscr{P})$ is defined by
	$$B^{(k)}:=\bigcup_{\tau:\dim(\tau)=k}\tau.$$
	If $\msc{P}$ is a toric polyhedral decomposition $\msc{P}$, then for each $\tau\in\msc{P}$, one defines the fan $\Sigma_{\tau}$ as the collection of the cones
	$$K_{\tau}(\sigma):=\bb{R}_{\geq 0}\cdot S_{\tau}(\sigma),$$
	where $\sigma$ runs over all elements in $\msc{P}$ such that $\tau\subset\sigma$ and $\text{Int}(\sigma)\cap U_{\tau}\neq\emptyset$. For a point $y\in \text{Int}(\tau)\backslash\Gamma$, we put
	$$\cu{Q}_{\tau}:=\cu{Q}_{\tau,y}:=\Lambda_y/\Lambda_{\tau,y},$$
	which can be identified with the lattice $N'$ in Condition (5) in Definition \ref{def:poly_decomp}. These lattices define a sheaf $\cu{Q}_{\msc{P}}$ on $B$.

	\begin{definition}[\cite{GS1}, Definition 1.43]\label{def:piecewise_linear}
		Let $B$ be an integral affine manifold with singularities and $\mathscr{P}$ a polyhedral decomposition of $B$. Let $U\subset B$ be an open set. An \emph{integral piecewise linear function on $U$} is a continuous function $\varphi$ so that $\varphi$ is integral affine on $U\cap \text{Int}(\sigma)$ for any top dimensional cell $\sigma\in\msc{P}$, and for any $y\in U\cap \text{Int}(\tau)$ (for some $\tau\in\msc{P}$), there exists a neighborhood $V$ of $y$ and $f\in\Gamma(V,\cu{A}ff(B,\bb{Z}))$ such that $\varphi=f$ on $V\cap \text{Int}(\tau)$. We denote the sheaf of integral piecewise linear functions on $B$ by $\cu{PL}_{\mathscr{P}}(B,\bb{Z})$.
	\end{definition}
	
	There is a natural inclusion $\cu{A}ff(B,\bb{Z})\hookrightarrow\cu{PL}_{\mathscr{P}}(B,\bb{Z})$, and we let $\cu{MPL}_{\mathscr{P}}$ be the quotient:
	$$0\to \cu{A}ff(B,\bb{Z})\to\cu{PL}_{\mathscr{P}}(B,\bb{Z})\to\cu{MPL}_{\mathscr{P}}\to 0.$$
	Locally, an element $\varphi\in\Gamma(B,\cu{MPL}_{\mathscr{P}})$ is a collection of piecewise linear functions $\{\varphi_U\}$ so that on each overlap $U\cap V$, the difference
	$$\varphi_U|_{B_0}-\varphi_V|_{B_0}$$
	is an integral affine function on $U\cap V\cap B_0$.
	
	\begin{definition}[\cite{GS1}, Definition 1.45]
		The sheaf $\cu{MPL}_{\mathscr{P}}$ is called the \emph{sheaf of multi-valued piecewise linear functions of the pair $(B,\msc{P})$}.
	\end{definition}
	
	The sheaf $\cu{MPL}_{\mathscr{P}}$ also fits into the following exact sequence:
	$$0\to i_*\Lambda^*\to\cu{PL}_{\msc{P}}(B,\bb{Z})/\bb{Z}\to\cu{MPL}_{\mathscr{P}}\to 0,$$
	where $\Lambda\subset TB_0$ is the lattice inherited from the integral structure and $\Lambda^*\subset T^*B_0$ is the dual lattice.
	
	\begin{definition}[\cite{GS1}, Definition 1.46]
		For each element $\varphi\in H^0(B,\cu{MPL}_{\mathscr{P}})$, its image under the connecting map $c_1:H^0(B,\cu{MPL}_{\mathscr{P}})\to H^1(B,i_*\Lambda^*)$ is called the \emph{first Chern class} of $\varphi$.
	\end{definition}

	\begin{definition}[\cite{GS1}, Definitions 1.47]
		A section $\varphi\in H^0(B,\cu{MPL}_{\mathscr{P}})$ is said to be \emph{(strictly) convex} if for any vertex $\{v\}\in\msc{P}$, there is a neighborhood $U\subset B$ of $v$ such that there is a (strictly) convex representative $\varphi_i$. 
	\end{definition}
	
	\begin{definition}[\cite{GS1}, Definition 1.48]
		A toric polyhedral decomposition $\msc{P}$ is said to be \emph{regular} if there exists a strictly convex multi-valued piecewise linear function $\varphi\in H^0(B,\cu{MPL}_{\mathscr{P}})$.
	\end{definition}
	
	\begin{assumption}
		All polyhedral decompositions in this paper are assumed to be regular; in particular, they are integral and toric.
	\end{assumption}
	
	Given a regular polyhedral decomposition $(\msc{P},\varphi)$, one can obtain another affine manifold with singularities $\check{B}$ together with a regular polyhedral decomposition $(\check{\msc{P}},\check{\varphi})$ by taking the dual cell of each cell in $\msc{P}$. We will not give the precise construction here but let us mention some facts about $(\check{B},\check{\msc{P}},\check{\varphi})$. Topologically, $\check{B}$ is same as $B$ and their singular loci coincide. However, their affine structures and monodromies around the singular loci are dual to each other. See \cite{GS1}, Section 1.4 for the precise construction of $(\check{B},\check{\msc{P}},\check{\varphi})$.
	
	\begin{definition}[cf. \cite{GS1}, Propositions 1.50 \& 1.51]
		The triple $(\check{B},\check{\msc{P}},\check{\varphi})$ is called the \emph{discrete Legendre transform of $(B,\msc{P},\varphi)$}.
	\end{definition}
	
	We will need the following notion later.
	
	\begin{definition}\label{def:elementary}
		A regular polyhedral decomposition $\msc{P}$ is called \emph{elementary} if for any cell $\sigma\in\msc{P}$, the dual cell $\check{\sigma}$ is an elementary simplex.
	\end{definition}

	\subsection{Toric degenerations}\label{sec:toric_deg}
	
	In \cite{GS1}, Gross and Siebert defined a \emph{toric degeneration} (of Calabi-Yau varieties) as a flat family $\check{p}:\check{\cu{X}}\to \Delta = \text{Spec}(\bb{C}[[t]])$ such that generic fiber of $\check{p}$ is smooth and the central fiber $\check{X}_0:=\check{p}^{-1}(0)$ is a union of toric varieties, intersecting along toric strata. By gluing the fan of each toric piece, they obtained an affine manifold with singularities $\check{B}$ together with a polyhedral decomposition $\check{\msc{P}}$. When the family $\check{\cu{X}}$ is polarized, the resulting polyhedral decomposition is regular, so there is a strictly convex multi-valued piecewise linear function $\check{\varphi}$ on $\check{B}$, giving rise to the discrete Legendre transform $(B,\msc{P},\varphi)$ of $(\check{B},\check{\msc{P}},\check{\varphi})$.
	The important \emph{reconstruction problem in mirror symmetry} is asking whether one can construct another toric degeneration $p:\cu{X}\to \Delta$ (which acts as a mirror family) from $(B,\msc{P},\varphi)$.
	
	In this section, we review the fan construction of the algebraic spaces associated to $(B,\msc{P})$ in \cite{GS1}; with a good choice of gluing data, such spaces serve as the central fibers of the toric degenerations $p:\cu{X}\to \Delta$ and $\check{p}:\check{\cu{X}}\to \Delta$.\footnote{For the cone construction, please refer to \cite[Section 2.1]{GS1}.}
	
	\subsubsection{The fan construction}
	
	We first recall the category $Cat(\msc{P})$, whose objects are given by elements of $\msc{P}$. To define morphisms, let $Bar(\msc{P})$ be the barycentric subdivision of $\msc{P}$, and set
	$$Hom(\tau,\sigma) := 
	\begin{cases}
	\emptyset & \text{ if }\tau\nsubseteq\sigma,\\
	\{id\} & \text{ if }\tau=\sigma.
	\end{cases}$$
	For the case $\tau\subset\sigma$, $Hom(\tau,\sigma)$ consists of all 1-simplices whose endpoints are the barycenters of $\tau$ and $\sigma$. For $e_1\in Hom(\tau,\sigma),e_2\in Hom(\sigma,\omega)$, the composition $e_2\circ e_1$ is defined as the third edge of the unique 2-simplex containing $e_1,e_2$.
	
	For each $\sigma\in\msc{P}$, the map $S_{\sigma}:U_{\sigma}\to\cu{Q}_{\sigma}$ defines a fan $\Sigma_{\sigma}$ on $\cu{Q}_{\sigma}$. Let $X_{\sigma}$ be the toric variety associated to $\Sigma_{\sigma}$. For $\tau\subset\sigma$, let
	$$\sigma^{-1}\Sigma_{\tau}:=\{K\in\Sigma_{\tau}:K\supset K_{\tau}(\sigma)\}.$$
	There is a fan projection $p_{\tau\sigma}:\sigma^{-1}\Sigma_{\tau}\to\Sigma_{\sigma}$, given by quotienting along the direction $\Lambda_{\sigma}$ and an inclusion $j_{\sigma\tau}:\sigma^{-1}\Sigma_{\tau}\to\Sigma_{\tau}$. There is a natural embedding $\iota_{\sigma\tau}:X_{\sigma}\to X(\sigma^{-1}\Sigma_{\tau})\subset X_{\tau}$ induced by the ring map
	$$z^m\mapsto\begin{cases}
	z^m & \text{ if }m\in K^{\vee}\cap K_{\tau}(\sigma)^{\perp}\cap\cu{Q}_{\sigma}^*=((K+\Lambda_{\sigma,\bb{R}})/\lambda_{\sigma,\bb{R}})^{\vee}\cap\cu{Q}_{\sigma}^*,
	\\0 & \text{ otherwise},
	\end{cases}$$
	for any $K\in\sigma^{-1}\Sigma_{\tau}$. For $e\in Hom(\tau,\sigma)$, we define the functor $F_A:{\bf{Cat}}(\msc{P})\to{\bf{Sch}}$ by $F(\tau) := X_{\tau}$ and $F(e):=\iota_{\sigma\tau}$ the natural inclusion.
	
	We can also twist the functor by certain gluing data. The barycentric subdivision $Bar(\msc{P})$ of $\msc{P}$ defines an open covering $\cu{W}:=\{W_{\tau}\}$ of $B$, where
	$$W_{\tau}:=\bigcup_{\substack{\sigma\in Bar(\msc{P})\\\sigma\cap \text{Int}(\tau)\neq\emptyset}}\text{Int}(\sigma).$$
	For $e\in Hom(\tau,\sigma)$, we define $W_e:=W_{\tau}\cap W_{\sigma}$.
	
	\begin{definition}[\cite{GS1}, Definition 2.10]
		Let $S$ be a scheme. A \emph{closed gluing data} (for the fan picture) for $\msc{P}$ over $S$ is C\v{e}ch 1-cocycle $s=(s_e)_{e\in\coprod_{\tau,\sigma\in\msc{P}}Hom(\tau,\sigma)}$ of the sheaf $\cu{Q}_{\msc{P}}\otimes_{\bb{Z}}\bb{G}_m(S)$ with respect to the cover $\cu{W}$ of $B$. Here, $s_e\in\Gamma(W_e,\cu{Q}_{\msc{P}}\otimes_{\bb{Z}}\bb{G}_m(S))=\cu{Q}_{\sigma}\otimes_{\bb{Z}}\bb{G}_m(S)$ for $e\in Hom(\tau,\sigma)$.
	\end{definition}
	
	The torus $\cu{Q}_{\sigma}\otimes_{\bb{Z}}\bb{G}_m(S)$ acts on $X_{\sigma}\times S$, so an element $s_e\in\cu{Q}_{\sigma}\otimes_{\bb{Z}}\bb{G}_m(S)$ gives an automorphism $s_e:X_{\sigma}\times S\to X_{\sigma}\times S$. We then obtain an $s$-twisted functor $F_{S,s}:{\mf{Cat}}(\msc{P})\to{\mf{Sch}}_S$ by setting
	\begin{align*}
	F_{S,s}(\tau) := X_{\tau}\times S,\, F_{S,s}(e) := (F(s)\times id_S)\circ s_e.
	\end{align*}
	We define
	$$X_0(B,\msc{P},s):=\lim_{\longrightarrow}F_{S,s}.$$
	
	In \cite{GS1}, Gross and Siebert introduced a special set of gluing data, which they called \emph{open gluing data} (for the fan picture). We will not go into details here (readers are referred to \cite[Definition 2.25]{GS1} for the precise definition), but it is essential for $X_0(B,\msc{P},s)$ to be the central fiber of some toric degeneration. Given such an open gluing data for $(B,\msc{P})$, one can associate a closed gluing data $s$ for the fan picture of $(B,\msc{P})$ (see \cite[Proposition 2.32]{GS1}). There is then an obstruction map $o:H^1(\cu{W},\cu{Q}_{\msc{P}}\otimes_{\bb{Z}}\bb{C}^{\times})\to H^2(B,\bb{C}^{\times})$ such that when $o(s)=1$, a projective scheme $X_0(B,\msc{P},s)$ can be constructed. 
	Throughout this paper, we assume that all closed gluing data
	are induced by open gluing data for the fan picture. We also assume that $A=\bb{C}$ and $S = \text{Spec}(\bb{C})$. One of the main result in \cite{GS1} is the following
	\begin{theorem}[\cite{GS1}, Theorem 5.2]
		Suppose $(B,\msc{P})$ is positive and simple and $s$ satisfies the (LC) condition in \cite[Proposition 4.25]{GS1}. Then there exists a log structure on $X_0(B,\msc{P},s)$ and a morphism $X_0(B,\msc{P},s)^{\dagger}\to Spec(\bb{C})^{\dagger}$ which is log smooth away from a subset $Z\subset X_0(B,\msc{P},s)$ of codimension 2.
	\end{theorem}
	
	\begin{remark}
		When $(B,\msc{P})$ (and hence $(\check{B},\check{\msc{P}}))$ is positive and simple, $X_0(B,\msc{P},s)$ and $\check{X}_0(\check{B},\check{\msc{P}},\check{s})$ are mirror to each other as log Calabi-Yau spaces in an appropriate sense; see \cite[Section 5.3]{GS1} for details.
	\end{remark}
	
	As aforementioned, the reconstruction problem in mirror symmetry is to construct a polarized toric degeneration $p:\cu{X}\to \Delta$ whose central fiber is given by $ X_0(B,\msc{P},s)$ for  some open gluing data $s$ over $\text{Spec}(\bb{C})$. This was solved by Gross and Siebert in \cite{GS11} using an explicit order-by-order algorithm and a key combinatorial object called \emph{scattering diagram} first introduced by Kontsevich and Soibelman in \cite{KS_scattering} (who first solved the reconstruction problem in dimension 2, but over non-Archimedean fields). When $(B,\msc{P})$ is \emph{positive} and \emph{simple}, Gross and Siebert proved that $ X_0(B,\msc{P},s)$ is always smoothable by writing down an explicit toric degeneration. 
	
	In \cite{GS2}, Gross and Siebert studied a specific type of logarithmic deformations, called \emph{divisoral deformations}. Similar to the classical deformation theory of schemes, the first order divisoral deformations of $ X_0(B,\msc{P},s)$ are parametrized by a first cohomology group $H^1( X_0(B,\msc{P},s),j_*\Theta_{ X_0(B,\msc{P},s)^{\dagger}/\bb{C}^{\dagger}})$, while the obstructions lie in the second cohomology group $H^2( X_0(B,\msc{P},s),j_*\Theta_{X_0(B,\msc{P},s)^{\dagger}/\bb{C}^{\dagger}})$ (see \cite[Theorem 2.11]{GS2}); here $\Theta_{ X_0(B,\msc{P},s)^{\dagger}/\bb{C}^{\dagger}}$ is the sheaf of \emph{logarithmic tangent vectors} of the log scheme $ X_0(B,\msc{P},s)^{\dagger}$ and $j: X_0(B,\msc{P},s)\backslash Z\hookrightarrow X_0(B,\msc{P},s)$ is the inclusion map. However, they did not prove existence of smoothings along this line of thought.
	
	Very recently, the first two authors of this paper and Leung \cite{CLM_smoothing}, by using gluing of local differential graded Batalin-Vilkovisky (dgBV) algebras and partly motivated by \cite{GS2}, developed an algebraic framework to prove existence of formal smoothings for singular Calabi-Yau varieties with prescribed local models. This covers the log smooth case studied by Friedman \cite{Friedman} and Kawamata-Namikawa \cite{Kawamata-Namikawa} as well as the maximally degenerate case studied by Kontsevich-Soibelman \cite{KS_scattering} and Gross-Siebert \cite{GS11}. More importantly, this approach provides a singular version of the classical Bogomolov-Tian-Todorov theory and bypasses the complicated scattering diagrams.
	
	This framework was subsequently applied by Felten-Filip-Ruddat in \cite{Ruddat_smoothing} to produce smoothings of a very general class of varieties called \emph{toriodal crossing spaces}.\footnote{Applying results in Ruddat-Siebert \cite{Ruddat_Siebert}, they were able to prove that the smoothings are actually analytic.} Such an algebraic framework should be applicable in a variety of settings, e.g. it was applied to smoothing of pairs in \cite{CM_pair}. In this paper, which can be regarded as a sequel to \cite{CM_pair}, we show how this approach can lead to an explicit, combinatorial construction of smoothable pairs in dimension 2.

	\section{Tropical Lagrangian multi-sections}\label{sec:trop_Lag}
	
	In this section, we introduce the notion of a tropical Lagrangian multi-section when $\dim(B)=2$.\footnote{From this point on, we will always assume that $\dim(B) = 2$.} These tropical objects should be viewed as limiting versions of Lagrangian multi-sections of a Lagrangian torus fibration (or SYZ fibration). In \cite{CS_SYZ_imm_Lag}, the first and third authors of this paper considered the case of a semi-flat Lagrangian torus fibration $X(B)\to B$, where a Lagrangian multi-section can be described by an unbranched covering map $\pi:L\to B$ together with a Lagrangian immersion into the symplectic manifold $X(B)$. However, in general (e.g., when $B$ is simply connected), the covering map $\pi:L\to B$ would have non-empty branch locus. We begin by describing what kind of covering maps is allowed.
	
	\begin{definition}
		Let $B$ be a 2-dimensional integral affine manifold with singularities equipped with a polyhedral decomposition $\msc{P}$. Let $L$ be a topological manifold. A $r$-fold topological covering map $\pi:L\to B$ with branch locus $S$ is called \emph{admissible} if
		\begin{enumerate}
			\item $S\subset B^{(0)}$, and
			\item $\pi^{-1}(B_0\backslash S)$ is an integral affine manifold such that $\pi|_{\pi^{-1}(B_0\backslash S)}:\pi^{-1}(B_0\backslash S)\to B_0\backslash S$ is an integral affine map.
		\end{enumerate}
		An admissible covering map $\pi:L\to B$ is said to have \emph{simple branching} if it satisfies the following extra condition:
		\begin{enumerate}[resume]
			\item For any $x\in B$, there exists a neighborhood $U\subset B$ of $x$ such that the preimage $\pi^{-1}(U)$ can be written as
			$$U'\amalg\coprod_{i=1}^{r-2}U_i'$$
			for some open subsets $U',U_1',\dots, U_{r-2}'\subset L$ so that $\pi|_{U'}:U'\to U$ is a (possibly branched) $2$-fold covering map.
		\end{enumerate}
	\end{definition}
	
	Given an admissible covering map $\pi:L\to B$, the domain $L$ is naturally an integral affine manifold with singularities and the singular locus is given by $S'\amalg\pi^{-1}(\Gamma)$, where $S'\subset L$ is the ramification locus of $\pi:L\to B$. We need to distinguish these two singular sets combinatorially.  
	
	\begin{definition}\label{def:poly_decomp_cover}
		Let $B$ be a 2-dimensional integral affine manifold with singularities equipped with a polyhedral decomposition $\msc{P}$. Let $\pi:L\to B$ be an admissible covering map. A \emph{polyhedral decomposition of $\pi:L\to B$} is a collection $\msc{P}_{\pi}$ of closed subsets of $L$ (also called \emph{cells}) covering $L$ so that
		\begin{enumerate}
			\item $\pi(\sigma')\in\msc{P}$ for all $\sigma'\in\msc{P}_{\pi}$;
			\item for any $\sigma\in\msc{P}$, we have $$\pi^{-1}(\sigma)=\bigcup_{\sigma'\in\msc{P}_{\pi}:\pi(\sigma')=\sigma}\sigma'$$
			and if we define the \emph{relative interior} of $\sigma'$ to be
			$$\text{Int}(\sigma'):=\pi^{-1}(\text{Int}(\sigma))\cap\sigma',$$
			then $\pi|_{\text{Int}(\sigma')}:\text{Int}(\sigma')\to \text{Int}(\sigma)$ is a homeomorphism; and
			\item if $\pi(\sigma^{(\alpha)})=\pi(\sigma^{(\beta)})$ and $\sigma^{(\alpha)}\cap\sigma^{(\beta)}\neq\emptyset$, then $\sigma^{(\alpha)}\cap\sigma^{(\beta)}\subset S'$.
		\end{enumerate}
		We define
		$$\dim(\sigma'):=\dim(\sigma)$$
		for all $\sigma'\in\cu{P}_{\pi}$ with $\pi(\sigma')=\sigma\in\msc{P}$.
	\end{definition}
	
	Clearly, if $\msc{P}_{\pi}$ is a polyhedral decomposition of an admissible covering map $\pi:L\to B$, then
	$$\pi(\msc{P}_{\pi}):=\{\pi(\sigma') \mid \sigma'\in\msc{P}_{\pi}\}=\msc{P}.$$
	We denote the $k$-dimensional strata of $(L,\msc{P}_{\pi})$ by
	$$L^{(k)}:=\bigcup_{\sigma':\dim(\sigma')=k}\sigma'.$$

	\begin{example}\label{eg:rank2}
		Let $B$ be the boundary of the unit cube in $\bb{R}^3$. Let $L$ be a 2-torus. Then $L$ branches over $B$ at 4 points. Denote this covering map by $\pi:L\to B$. We can assume the branch points (colored points) are located as in Figure \ref{fig:example_rk2}. Vertices with same color are identified and each colored vertex on $L$ corresponds to a branch point on $B$ with the same color.
		\begin{figure}[ht]
			\includegraphics[width=110mm]{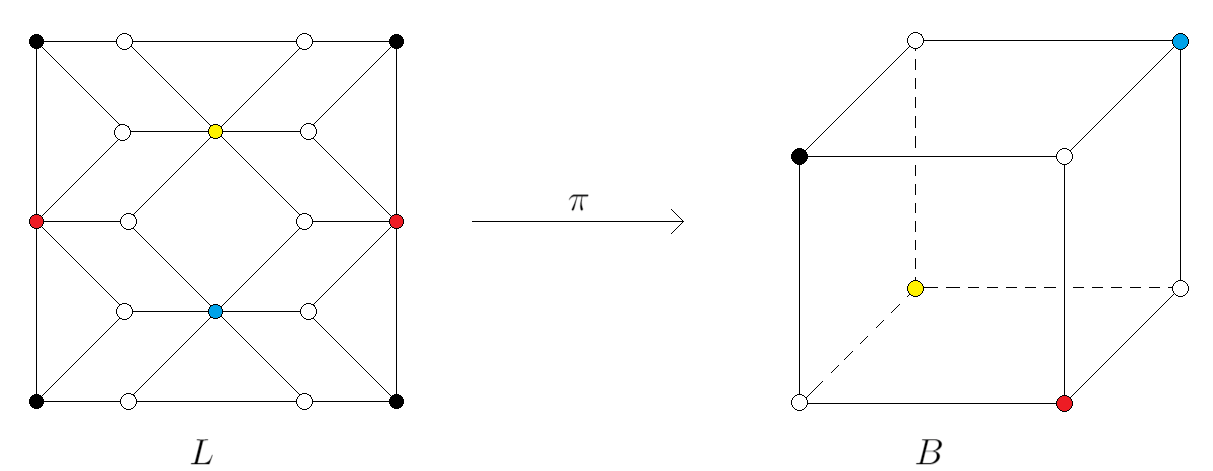}
			\caption{}
			\label{fig:example_rk2}
		\end{figure}
	\end{example}
	
	As the domain $L$ is an integral affine manifold with singularities, we can therefore define the sheaf of integral affine functions $\cu{A}ff(L,\bb{Z})$ on $L$ as before. However, since the singular locus $S'$ lies in $L^{(0)}$, we need to clarify what we mean by a piecewise linear function though it is similar to Definition \ref{def:piecewise_linear}:
	\begin{definition}
		Let $B$ be an integral affine manifold with singularities equipped with a polyhedral decomposition $\msc{P}$. Let $\pi:L\to B$ be an admissible covering map equipped with a polyhedral decomposition $\msc{P}_{\pi}$. Let $\til{U}$ be an open subset of $L$. A \emph{piecewise linear function on $\til{U}$} is a continuous function $\varphi'$ which is affine linear on $\text{Int}(\sigma')$ for any maximal cell $\sigma'\in\msc{P}_{\pi}$ and satisfies the following property: for any $y'\in\til{U}$ and $y'\in \text{Int}(\tau')$ (for some $\tau'\in\msc{P}_{\pi}$), there is a neighborhood $\til{V}$ of $y'$ and $f'\in\Gamma(\til{V},\cu{A}ff(L,\bb{Z}))$ such that
		$$\varphi'|_{\til{V}\cap\tau'}=f'|_{\til{V}\cap\tau'}.$$
		We denote the sheaf of piecewise linear functions on $L$ by $\cu{PL}_{\msc{P}_{\pi}}(L,\bb{Z})$.
	\end{definition}
	
	\begin{definition}
		The \emph{sheaf of multi-valued piecewise linear functions} $\cu{MPL}_{\mathscr{P}_{\pi}}$ on $L$ is defined as the quotient
		$$0\to \cu{A}ff(L,\bb{Z})\to\cu{PL}_{\mathscr{P}_{\pi}}(L,\bb{Z})\to\cu{MPL}_{\mathscr{P}_{\pi}}\to 0.$$
	\end{definition}
	
	We are now ready to define the main object to be studied in this paper.
	\begin{definition}\label{def:trop_lag}
		Let $B$ be a 2-dimensional integral affine manifold with singularities equipped with a polyhedral decomposition $\mathscr{P}$. A \emph{tropical Lagrangian multi-section of rank $r$} is a quadruple $\bb{L}:=(L,\pi,\msc{P}_{\pi},\varphi')$, where
		\begin{enumerate}
			\item [1)] $L$ is a topological manifold and $\pi:L\to B$ is an admissible $r$-fold covering map,
			\item [2)] $\msc{P}_{\pi}$ is a polyhedral decomposition of $\pi:L\to B$, and
			\item [3)] $\varphi'$ is a multi-valued piecewise linear function on $L$.
		\end{enumerate}
	\end{definition}

	\section{A local model around the ramification locus}\label{sec:rk2_trop_lag}\label{sec:local_model}
	
	In this section, we prescribe a local model for $\bb{L}$ around each ramification point of $\pi:L\to B$; such a model is motivated by the work \cite{branched_cover_fan} of Payne and the previous work \cite{Suen_TP2} of the third author. We also give an explicit construction of a certain class of rank 2 tropical Lagrangian multi-sections over $B$.
	
	\begin{definition}
		Let $B$ be an integral affine manifold with singularities equipped with a polyhedral decomposition $\msc{P}$, and $\pi:L\to B$ an admissible covering map. A branch point $v\in S\subset B$ is called \emph{standard} if there is an isomorphism $\Sigma_v\cong\Sigma_{\bb{P}^2}$ between the fan $\Sigma_v$ and that of $\bb{P}^2$.
	\end{definition}
	
	For a standard vertex $v\in S$, we construct a fan $\Sigma_{v'}$, a $2$-fold covering map $\pi_*:|\Sigma_{v'}|\to|\Sigma_v|$ with a unique branch point at $0$ and a piecewise linear function $\varphi_{v'}:|\Sigma_{v'}|\to\bb{R}$ on $|\Sigma_{v'}|$ as follows.
	
	Let $K_{\sigma_0},K_{\sigma_1},K_{\sigma_2}\in\Sigma_v$ be the top dimensional cones of $\Sigma_v$ which correspond to the cones $\sigma_0,\sigma_1,\sigma_2 \in \Sigma_{\bb{P}^2}$ (see Figure \ref{fig:L}) respectively. Let $K_{\sigma_i}^{\pm}$ be two copies of $K_{\sigma_i}$. Let $\rho_j,\rho_k$ be the rays spanning $K_{\sigma_i}$ and $\rho_j^{\pm},\rho_k^{\pm}$ be that for $K_{\sigma_i}^{\pm}$, for $i,j,k \in \{0,1,2\}$ being distinct. We glue $K_{\sigma_0}^{\pm}$ with $K_{\sigma_1}^{\mp}$ and $K_{\sigma_2}^{\mp}$ by identifying $\rho_1^{\pm}$ with $\rho_1^{\mp}$ and $\rho_2^{\pm}$ with $\rho_2^{\mp}$ respectively, and glue $K_{\sigma_1}^{\pm}$ with $K_{\sigma_2}^{\mp}$ by identifying $\rho_0^{\pm}$ with $\rho_0^{\mp}$. Then the fan $\Sigma_{v'}$ is given by
	$$\{K_{\sigma_i}^{\pm},\rho_i^{\pm},0 \mid i=0,1,2\}.$$
	The projection $\pi_*:|\Sigma_{v'}|\to|\Sigma_v|$ is defined by $K_{\sigma_i}^{\pm}\to K_{\sigma_i}$.
	
	On the fan $\Sigma_{v'}$, one can define a piecewise linear function $\varphi_{v'}$ by setting
	$$\varphi_{v'}:=\begin{cases}
	(n-m)\xi_1 &\text{ on }K_{\sigma_0}^+,\\
	(n-m)\xi_2 &\text{ on }K_{\sigma_0}^-,\\
	n\xi_1 &\text{ on }K_{\sigma_1}^+,\\
	n\xi_1+(n-m)\xi_2 &\text{ on }K_{\sigma_1}^-,\\
	(n-m)\xi_1+n\xi_2 &\text{ on }K_{\sigma_2}^+,\\
	n\xi_2 &\text{ on }K_{\sigma_2}^-,
	\end{cases}$$
	for some $m,n\in\bb{Z}$ with $m\neq n$. Here, $\xi_1, \xi_2$ are affine coordinates on $|\Sigma_v|\cong\bb{R}^2$. This gives a tropical Lagrangian multi-section $(|\Sigma_{v'}|,\pi_*,\Sigma_{v'},\varphi_{v'})$ of the smooth affine manifold $|\Sigma_v|$ with polyhedral decomposition given by the fan $\Sigma_v$.
	
	Recall that when $\pi$ has simple branching, for any branch point $v\in S$, one can choose a neighborhood $U$ of $v$ such that
	$$\pi^{-1}(U)=U'\amalg\coprod_{i=1}^{r-2}U_i',$$
	for some open subsets $U',U_1',\dots, U_{r-2}'\subset L$ and $\pi|_{U'}:U'\to U$ is a branched 2-fold covering map. Let $\msc{P}_{\pi}$ be a polyhedral decomposition of $\pi:L\to B$. Denote by $\msc{P}_{v'}$ the collection of all cells $\sigma'\in\msc{P}_{\pi}$ such that $v'\in\sigma'$. Write
	$$\pi^{-1}(\sigma\cap U)\cap U'=(\sigma^+\cup\sigma^-)\cap U'$$
	for $\sigma\in\msc{P}$ and $\sigma^{\pm}\in\msc{P}_{v'}$.

	\begin{definition}\label{def:proj_Lag}
		Let $B$ be a 2-dimensional integral affine manifold with singularities, $\msc{P}$ a polyhedral decomposition of $B$. A tropical Lagrangian multi-section $\bb{L}:=(L,\pi,\msc{P}_{\pi},\varphi')$ is said to be \emph{of class $\mathscr{S}$}, denoted as $\bb{L}\in\mathscr{S}$, if 
		\begin{enumerate}
			\item every branch point of $\pi:L\to B$ is standard;
			\item for any vertex $v\in S$, there exists a neighborhood $U\subset B$ of $v$ and two integral affine embeddings $f':U'\to|\Sigma_{v'}|$, $f:U\to|\Sigma_v|$ such that $f'(\sigma^{\pm}\cap U')\subset K_{\sigma}^{\pm}$ for all $\sigma\in\msc{P}$ with $v\in\sigma$ and $\pi_*\circ f'=f\circ\pi$; and
			\item for each vertex $v\in S$, $\varphi'$ can be represented by $\varphi_{v'}$ on $U'$, i.e., if $\varphi'_{U'}$ is a representative of $\varphi'$ on $U'$, then
			$$\varphi_{U'}'-\varphi_{v'}|_{f'(U')}\circ f'\in \cu{A}ff(U',\bb{Z}),$$
			for some $m(v'),n(v')\in\bb{Z}$ defining $\varphi_{v'}$.
		\end{enumerate}
	\end{definition}

	For \emph{fixed} $m,n\in\bb{Z}$ with $m\neq n$, there is a subclass $\msc{S}_{m,n}$ of $\msc{S}$ defined as
	$$\msc{S}_{m,n}:=\{\bb{L}\in\msc{S}:\{m(v'),n(v')\}=\{m,n\}\text{ for all }v'\in S'\}.$$
	This collection of tropical Lagrangian multi-sections will be our main object of study in Section \ref{sec:simple_smooth}.

	\section{Construction of $\cu{E}_0(\bb{L},{\bf{k}}_s)$ for $\bb{L} \in \mathscr{S}$}\label{sec:E_0}
	
	Let $(B,\msc{P})$ be a 2-dimensional integral affine manifold with singularities equipped with a regular polyhedral decomposition $\msc{P}$. Let $\bb{L}$ be a rank $r$ tropical Lagrangian multi-section of class $\mathscr{S}$ as in Definition \ref{def:proj_Lag}. The goal of this section is to construct a locally free sheaf $\cu{E}_0(\bb{L},{\bf{k}}_s)$ of rank $r$ on the scheme $X_0(B,\msc{P},s)$ associated to $\bb{L}$ equipped with some continuous data ${\bf{k}}_s$ which depends on the gluing data $s$.
	
	First of all, for any vertex $\{v\}\in\msc{P}$, we need to construct a locally free sheaf $\cu{E}(v)$ over the toric component $X_v \subset X_0(B,\msc{P},s)$. Let $\sigma_1,\dots,\sigma_l$ be the top dimensional cells that contain $v$ and $U_{\sigma_i}\subset X_v$ be the toric affine chart corresponding to $\sigma_i$. Then $\{U_{\sigma_i}\}$ forms an open covering of $X_v$.
	Suppose the vertex $v\notin S$. Then the preimage of $v$ consists of $r$ distinct points $v^{(\alpha)}$, for $k=1,\dots,r$. Hence we obtain $r$ piecewise linear functions $\varphi_{v^{(\alpha)}}:|\Sigma_{v^{(\alpha)}}|\to\bb{R}$, which correspond to $r$ toric line bundles $\cu{L}(v^{(\alpha)})$ over $X_v$. The transition function of $\cu{L}(v^{(\alpha)})$ on the overlap $U_v(\sigma_1)\cap U_v(\sigma_2)$ can be chosen as
	$$z^{m_k(\sigma_1')-m_k(\sigma_2')},$$
	where $m_k(\sigma')$ is the slope of $\varphi_{v^{(\alpha)}}$ on the maximal cell $\sigma'$. In this case, we put
	$$\cu{E}(v):=\bigoplus_{k=1}^r\cu{L}(v^{(\alpha)}).$$
	Now suppose $v\in S$. Then by our assumption, there are precisely three top dimensional cells $\sigma_0,\sigma_1,\sigma_2$ containing $v$. In this case, $X_v\cong\bb{P}^2$. Let $v'\in S'$ be the unique ramification point such that $\pi(v')=v$. In a neighborhood $U$ of $v$, write 
	$$\pi^{-1}(\sigma_i\cap U):=\left((\sigma_i^+\cup\sigma_i^-)\cap U'\right)\amalg\coprod_{k=1}^{r-2}\sigma_i^{(\alpha)}\cap U_k'$$
	with $\sigma_i^+\cap\sigma_i^-=\{v'\}$ being the ramification point. As $\varphi|_{U'}$ can be represented by $\varphi_{v'}$, by restricting to $\sigma_i^{\pm},\sigma_i^{(\alpha)}$, we obtain $r$ integral affine functions. Hence they define $r$ toric line bundles $\cu{L}_i^{\pm},\cu{L}_i^{(\alpha)}$, $k=1,\dots,r-2$, on the affine chart $U_i:=U_v(\sigma_i)$, where
	\begin{align*}
	\cu{L}_0^+=\,&\cu{O}((n-m)D_1)|_{U_0},\quad\cu{L}_0^-=\cu{O}((n-m)D_2)|_{U_0},\\
	\cu{L}_1^+=\,&\cu{O}(nD_0)|_{U_1},\quad\quad\quad\quad\cu{L}_1^-=\cu{O}((m-n)D_2+(2n-m)D_0)|_{U_1},\\
	\cu{L}_2^+=\,&\cu{O}(nD_0)|_{U_2},\quad\quad\quad\quad\cu{L}_2^-=\cu{O}((m-n)D_1+(2n-m)D_0)|_{U_2}.
	\end{align*}
	Here, $U_i\subset\bb{P}^2$ are the affine charts corresponding to the cones $K_{\sigma_i}$ and $D_k$ is the divisor corresponding to the ray $\rho_k$. In this case, we set
	$$\cu{E}_i:=(\cu{L}_i^+\oplus\cu{L}_i^-)\oplus\bigoplus_{k=1}^{r-2}\cu{L}_i^{(\alpha)},$$
	which is a rank $r$ vector bundle on $U_{\sigma_i}$.
	
	Next, we try to glue the bundles $\cu{E}_i$'s together. First of all, the line bundles $\{\cu{L}_i^{(\alpha)}\}_{i=0,1,2}$ naturally glue together to form a toric line bundle $\cu{L}(v_k')$ on $ X_v$ as before. 
	
	However, in the case $v\in S$, the rank 2 bundles $\{\cu{L}_i^+\oplus\cu{L}_i^-\}_{i=0,1,2}$ cannot be glued equivariantly. This is because when we try to glue $\cu{L}_i^{\pm}$ to $\cu{L}_j^{\mp}$ equivariantly, the gluing data consists of two diagonal matrices and one off-diagonal matrix, and so the cocycle condition fails to hold. More precisely, the equivariant structure on each $\cu{L}_i^+\oplus\cu{L}_i^-$ is given by the minus of the slopes of $\varphi_{v'}$ on $K_{\sigma_i}$:
	\begin{equation}\label{eqn:sf_equ_str}
	\begin{split}
	(\lambda_1,\lambda_2)\cdot(w_0^1,w_0^2,v^+_0,v^-_0) & = (\lambda_1w_0^1,\lambda_2w_0^2,\lambda_1^{m-n}v^+_0,\lambda_2^{m-n}v^-_0),
	\\(\lambda_1,\lambda_2)\cdot(w_1^0,w_1^2,v^-_1,v^+_1) & = (\lambda_1^{-1}w_1^0,\lambda_1^{-1}\lambda_2w_1^2,\lambda_1^{-n}v^-_1,\lambda_1^{-n}\lambda_2^{m-n}v^+_1),
	\\(\lambda_1,\lambda_2)\cdot(w_2^0,w_2^1,v^+_2,v^-_2) & = (\lambda_2^{-1}w_2^0,\lambda_1\lambda_2^{-1}w_2^1,\lambda_2^{-n}v_2^+,\lambda_1^{m-n}\lambda_2^{-n}v^-_2);
	\end{split}
	\end{equation}
	here $v_i^{\pm}$ are fiber coordinates of $\cu{L}_i^{\pm}$. According to the gluing of $|\Sigma_{v'}|\backslash\{v'\}$, one can write down the naive transition functions
	$$\tau^{sf}_{10}:=\begin{pmatrix}
	\frac{a_0}{(w_0^1)^m} & 0
	\\0 & \frac{b_0}{(w_0^1)^n}
	\end{pmatrix},
	\tau^{sf}_{21}:=\begin{pmatrix}
	\frac{b_1}{(w_1^2)^n} & 0
	\\0 & \frac{a_1}{(w_1^2)^m}
	\end{pmatrix},
	\tau^{sf}_{02}:=\begin{pmatrix}
	0 & \frac{b_2}{(w_2^0)^n}
	\\\frac{a_2}{(w_2^0)^m} & 0
	\end{pmatrix},$$
	where $w_i^j:=\frac{\zeta^j}{\zeta^i}$ are inhomogeneous coordinates of a point $[\zeta^0:\zeta^1:\zeta^2]\in\bb{P}^2$ and $a_i,b_i$ are arbitrary non-zero constants. It is clear that they do not compose to the identity.
	
	To correct the gluing, we introduce three automorphisms. For $m>n$, we consider
	\begin{align*}
	\Theta_{10}& := I+\begin{pmatrix}
	0 & 0
	\\-a_0b_1a_2\left(\frac{w^2_0}{w^1_0}\right)^{m-n} & 0
	\end{pmatrix}\in Aut\left(\cu{O}|_{U_{10}}\oplus\cu{O}|_{U_{10}}\right),
	\\\Theta_{21}&:=I+\begin{pmatrix}
	0 & -a_0a_1b_2\left(\frac{w^0_1}{w^2_1}\right)^{m-n}
	\\0  & 0
	\end{pmatrix}\in Aut\left(\cu{O}(mD_0)|_{U_{21}}\oplus\cu{O}(nD_0)|_{U_{21}}\right),
	\\\Theta_{02}&:=I+\begin{pmatrix}
	0 & 0
	\\-b_0a_1a_2\left(\frac{w^1_2}{w^0_2}\right)^{m-n} & 0
	\end{pmatrix}\in Aut\left(\cu{O}(mD_0)|_{U_{02}}\oplus\cu{O}(nD_0)|_{U_{02}}\right),
	\end{align*}
	while for $m<n$, we consider
	\begin{align*}
	\Theta_{10}&:=I+\begin{pmatrix}
	0 & -b_0a_1b_2\left(\frac{w^2_0}{w^1_0}\right)^{n-m}
	\\0 & 0
	\end{pmatrix}\in Aut\left(\cu{O}|_{U_{10}}\oplus\cu{O}|_{U_{10}}\right),
	\\\Theta_{21}&:=I+\begin{pmatrix}
	0 & 0
	\\-b_0b_1a_2\left(\frac{w^0_1}{w^2_1}\right)^{n-m}  & 0
	\end{pmatrix}\in Aut\left(\cu{O}(mD_0)|_{U_{21}}\oplus\cu{O}(nD_0)|_{U_{21}}\right),
	\\\Theta_{02}&:=I+\begin{pmatrix}
	0 & -a_0b_1b_2\left(\frac{w^1_2}{w^0_2}\right)^{n-m}
	\\0 & 0
	\end{pmatrix}\in Aut\left(\cu{O}(mD_0)|_{U_{02}}\oplus\cu{O}(nD_0)|_{U_{02}}\right).
	\end{align*}
	The factors $\Theta_{ij}$ are written in terms of the frame of $\cu{E}_j$ on $U_j$. Let
	$$\tau_{ij}:=\tau_{ij}^{sf}\Theta_{ij}.$$
	Then a straightforward computation gives the following
	\begin{proposition}\label{prop:monodromy_free}
		If we impose the condition that $\prod_ia_ib_i=-1$, then
		$$\tau_{02}\tau_{21}\tau_{10}=I.$$
		Moreover, the equivariant structure defined by (\ref{eqn:sf_equ_str}) can be extended.
	\end{proposition}
	
	Hence we obtain a rank 2 toric vector bundle $E_{m,n}$ on $X_v\cong\bb{P}^2$ (for $v \in S$). The mysterious constants $a_i,b_i$'s are indeed irrelevant to the holomorphic structure.
	
	\begin{lemma}\label{lem:indep}
		The holomorphic structure of $E_{m,n}$ is independent of $a_i,b_i$'s as long as $\prod_ia_ib_i=-1$.
	\end{lemma}
	\begin{proof}
		We will only consider the case $m>n$; the proof for $n>m$ is similar. Let $\tau_{ij}'$'s be the transition functions of $E_{m,n}$ with $a_i=-1$, $b_i=1$, for all $i=0,1,2$. We define $f$ by
		\begin{align*}
		f|_{U_0}:=f_0:=& \begin{pmatrix}
		1 & 0
		\\0 & a_0b_1a_2
		\end{pmatrix},
		\\f|_{U_1}:=f_1:=& \begin{pmatrix}
		-a_0 & 0
		\\0 & -a_1^{-1}b_2^{-1}
		\end{pmatrix},
		\\f|_{U_2}:=f_2:=& \begin{pmatrix}
		a_0b_1 & 0
		\\0 & b_2^{-1}
		\end{pmatrix}.
		\end{align*}
		Using $\prod_ia_ib_i=-1$, one can check that
		$$\tau_{02}f_2=f_0\tau_{02}',\quad \tau_{21} f_1=f_2\tau_{21}',\quad \tau_{10}f_0=f_1\tau_{10}'.$$
		Hence $f$ defines an isomorphism.
	\end{proof}
	From now on, we will take $a_i=-1,b_i=1$, for all $i=0,1,2$. Then the corresponding transition functions are given by
	\begin{equation}\label{eqn:transition_function}
	\tau_{10}':=
	\begin{pmatrix}
	-\frac{1}{(w_0^1)^m} & 0
	\\-\frac{(w_0^2)^{m-n}}{(w_0^1)^m} & \frac{1}{(w_0^1)^n}
	\end{pmatrix},
	\tau_{21}':=
	\begin{pmatrix}
	\frac{1}{(w_1^2)^n} & -\frac{(w_1^0)^{m-n}}{(w_1^2)^m}
	\\0 & -\frac{1}{(w_1^2)^m}
	\end{pmatrix},
	\tau_{02}':=
	\begin{pmatrix}
	-\frac{(w_2^1)^{m-n}}{(w_2^0)^m} & \frac{1}{(w_2^0)^n} \\-\frac{1}{(w_2^0)^m} & 0
	\end{pmatrix},
	\end{equation}
	for $m>n$ and 
	$$\tau_{10}':=
	\begin{pmatrix}
	-\frac{1}{(w_0^1)^m} & -\frac{(w_0^2)^{n-m}}{(w_0^1)^n}
	\\0 & \frac{1}{(w_0^1)^n} 
	\end{pmatrix},
	\tau_{21}':=
	\begin{pmatrix}
	\frac{1}{(w_1^2)^n} & 0
	\\-\frac{(w_1^0)^{n-m}}{(w_1^2)^n} & -\frac{1}{(w_1^2)^m}
	\end{pmatrix},
	\tau_{02}':=
	\begin{pmatrix}
	0 & \frac{1}{(w_2^0)^n}
	\\-\frac{1}{(w_2^0)^m} & -\frac{(w_2^1)^{m-n}}{(w_2^0)^m} 
	\end{pmatrix},$$
	for $n>m$.
	\begin{proposition}\label{prop:dual_isom}
		For any $m,n\in\bb{Z}$ with $m\neq n$, we have $E_{m,n}\cong E_{n,m}$ and $E_{m,n}^*\cong E_{-m,-n}$.
	\end{proposition}
	\begin{proof}
		Let $\{\tau_{ij}^{m,n}\}$ be the transition functions of $E_{m,n}$. Put
		$$J:=\begin{pmatrix}
		0 & -1
		\\1 & 0
		\end{pmatrix}.$$
		Then we have
		$$\tau_{10}^{m,n}J=J^{-1}\tau_{10}^{n,m},\quad \tau_{21}^{m,n}J^{-1}=J\tau_{21}^{n,m},\quad \tau_{02}^{m,n}J=J\tau_{02}^{n,m}.$$
		Hence $E_{m,n}\cong E_{n,m}$. It is immediate that $(\tau_{ij}^{m,n})^{-1}=(\tau_{ij}^{-m,-n})^t$. Thus $E_{m,n}^*\cong E_{-m,-n}$.
	\end{proof}

	Before constructing the sheaf $\cu{E}_0(\bb{L},{\bf{k}}_s)$, we first prove the simplicity of the rank 2 bundle $E_{m,n}$. This will be crucial in the study of smoothability of the pair $(X_0(B,\msc{P},s),\cu{E}_0(\bb{L},{\bf{k}}_s))$ in Section \ref{sec:simple_smooth}. Recall that a locally free sheaf $\cu{E}$ on a scheme $X$ is called \emph{simple} if $H^0(X,\cu{E}nd(\cu{E}))=\bb{C}$, or equivalently, $H^0(X,\cu{E}nd_0(\cu{E}))=0$ where $\cu{E}nd_0$ denotes the sheaf of traceless endomorphisms.
	
	For this purpose, we compute the Chern classes of $E_{m,n}$. The piecewise linear function $\varphi_{m,n}$ allows us to obtain the equivariant Chern classes as piecewise polynomial functions on $|\Sigma_{\bb{P}^2}|$ (see \cite[Section 3.2]{branched_cover_fan}) , namely,
	\begin{align*}
	c_1^{(\bb{C}^{\times})^2}(E_{m,n})=&\begin{cases}
	(n-m)(\xi_1+\xi_2) & \text{ on }\sigma_0,
	\\2n\xi_1+(n-m)\xi_2 & \text{ on }\sigma_1,
	\\(n-m)\xi_1+2n\xi_2 & \text{ on }\sigma_2.
	\end{cases}
	\\c_2^{(\bb{C}^{\times})^2}(E_{m,n})=&\begin{cases}
	(n-m)^2\xi_1\xi_2 & \text{ on }\sigma_0,
	\\n^2\xi_1^2+n(n-m)\xi_1\xi_2 & \text{ on }\sigma_1,
	\\n(n-m)\xi_1\xi_2+n^2\xi_2^2 & \text{ on }\sigma_2.
	\end{cases}
	\end{align*}
	Since the equivariant cohomology $H^{\bullet}_{(\bb{C}^{\times})^2}(\bb{P}^2;\bb{Z})$ of $\bb{P}^2$ is given by $\bb{Z}[t_0,t_1,t_2]/(t_0t_1t_2)$, where $t_i$'s are piecwewise linear functions on $|\Sigma_{\bb{P}^2}|$ so that $t_i(v_j)=\delta_{ij}$, and the forgetful map $H^{\bullet}_{(\bb{C}^{\times})^2}(\bb{P}^2;\bb{Z})\to H^{\bullet}(\bb{P}^2;\bb{Z})\cong\bb{Z}[H]/(H^3)$ is given by mapping all the $t_i$'s to the hyperplane class $H$, we have
	$$c(E_{m,n})=1+(m+n)H+(m^2+n^2-mn)H^2.$$
	
	\begin{proposition}\label{prop:simple}
		For all $m,n\in\bb{Z}$ with $m\neq n$, the bundle $E_{m,n}$ is stable and hence simple.
	\end{proposition}
	\begin{proof}
		A straightforward calculation shows that we have the equality
		$$c_2-4c_1^2=-3(m-n)^2.$$
		Hence $c_2-4c_1^2<0$ and $c_2-4c_1^2\neq -4$. These conditions are equivalent to stability of rank 2 bundles on $\bb{P}^2$; see \cite{Schwarzenberger}.
	\end{proof}
	
	Let us go back to the construction of $\cu{E}_0(\bb{L},{\bf{k}}_s)$. The fan structure around $v\in S$ defines a toric isomorphism $ X_v\cong\bb{P}^2$. Denote the corresponding bundle on $ X_v$ by $\cu{E}(v')$ and put
	$$\cu{E}(v):=\cu{E}(v')\oplus\bigoplus_{k=1}^{r-2}\cu{L}(v^{(\alpha)}),$$
	which is a rank $r$ bundle on $\bb{P}^2$.
	In this way, the tropical Lagrangian multi-section $\bb{L}$ defines a locally free sheaf $\cu{E}(\sigma)$ on $ X_{\sigma}$ for each $\sigma\in\msc{P}$. 
	Now we fix a closed gluing data $s = (s_e)$ for the fan picture. In order to obtain a consistent gluing, we need to find a set of isomorphisms $\{h_{\tau_2\tau_1,s}:\cu{E}(\tau_2)\to F_s(e)^*\cu{E}(\tau_1)\}_{e:\tau_1\to\tau_2}$ such that
	\begin{equation}\label{eqn:cocycle_g_s}
	(F_s(e_2)^*h_{\tau_2\tau_1,s})\circ h_{\tau_3\tau_2,s}=h_{\tau_3\tau_1,s}
	\end{equation}
	for $e_1:\tau_1\to\tau_2,e_2:\tau_2\to\tau_3$ and $e_3:=e_2\circ e_1$.
	
	We first construct such data for the case $s=1$. Let $e:\{v\}\to\tau$. If $v\notin S$, then there is a natural inclusion $h_{\tau\{v\}}:\cu{E}(\tau)\to F(e)^*\cu{E}(v)$, given by
	$$h_{\tau\{v\}}|_{U_{\tau}(\sigma)}:1_{\sigma}^{(\alpha)}(\tau)\mapsto F(e)^*1_{\sigma}^{(\alpha)}(v)$$
	on the affine chart $U_{\tau}(\sigma)\subset X_{\tau}$. If $v\in S$, note that the transition function of $\cu{E}(v')$ on $U_v(\sigma_1\cap\sigma_2)$ as $$\tau_{\sigma_1\sigma_2}(v):1_v(\sigma_1^{(\alpha)})\mapsto\sum_{\alpha=1}^2\epsilon_{\sigma_1\sigma_2}^{(\alpha\beta)}z^{m_v(\sigma_1^{(\alpha)})-m_v(\sigma_2^{(\beta)})}1_v(\sigma_2^{(\beta)}),$$
	where $m_v(\sigma_i^{(\alpha)})$ is the slope of the local representative of $\varphi_{v'}$ on $\sigma_i^{(\alpha)}$ and $\epsilon_{\sigma_1\sigma_2}^{(\alpha\beta)}\in\{0,1,-1\}$ are prescribed as in (\ref{eqn:transition_function}). Let $e:\{v\}\to\tau$, with $\tau$ being an edge. Note that $\Theta_{ij}\equiv Id$ on the divisor $D_k$, for $i,j,k=0,1,2$ being distinct. The transition maps $F(e)^*\tau_{\sigma_1\sigma_2}(v)$ of $F(e)^*\cu{E}(v)$ is then given by
	$$F(e)^*1_v(\sigma_1^{(\alpha)})\mapsto \epsilon_{\sigma_1\sigma_2}^{(\alpha\alpha')}(v)z^{m_v(\sigma_1^{(\alpha)})-m_v(\sigma_2^{(\alpha')})}F(e)^*1_v(\sigma_2^{(\alpha')}),$$
	where $\alpha'$ is uniquely determined by the conditions $v'\notin\sigma_1^{(\alpha)}\cap\sigma_2^{(\alpha')}$ and $\pi(\sigma_2^{(\alpha')})=\sigma_2$. Note that $\epsilon_{\sigma_1\sigma_2}^{(\alpha\alpha')}(v)\neq 0$ for all $\alpha$. Since $X_{\tau}$ is covered by two charts $U_{\tau}(\sigma_1),U_{\tau}(\sigma_2)$, we can define
	$$H(e):\begin{cases}
	1_{\tau}(\sigma_1^{(\alpha)})\mapsto
	F(e)^*1_v(\sigma_1^{(\alpha)}) & \text{ on }U_{\tau}(\sigma_1),\\
	1_{\tau}(\sigma_2^{(\alpha)})\mapsto\epsilon_{\sigma_1\sigma_2}^{(\alpha\alpha')}(v)F(e)^*1_v(\sigma_2^{(\alpha')}) & \text{ on }U_{\tau}(\sigma_2).
	\end{cases}$$
	Then it is easy to see that
	$$H_s(e)|_{U_{\tau}(\sigma_2)}\circ \tau_{\sigma_1\sigma_2}(\tau)=F(e)^*\tau_{\sigma_1\sigma_2}(v)\circ H_s(e)|_{U_{\tau}(\sigma_1)}.$$
	Hence $H(e):\cu{E}(\tau)\to F(e)^*\cu{E}(v)$ defines an isomorphism. For $e_3:\{v\}\xrightarrow{e_1}\tau\xrightarrow{e_2}\sigma$, with $\sigma\in\msc{P}(2)$, we put
	\begin{align*}
	H(e_2):&\,1_{\sigma}(\sigma^{(\alpha)})\mapsto (\epsilon_{\sigma_1\sigma_2}^{(\alpha\alpha')}(v))^{-1}F(e_2)^*1_{\tau}(\sigma^{(\alpha)}),\\
	H(e_3):&\,1_{\sigma}(\sigma^{(\alpha)})\mapsto F(e_3)^*1_v(\sigma^{(\alpha')}).
	\end{align*}
	Note that $H(e_2)$ is well-defined because $\sigma_1,\sigma_2$ are uniquely determined by $\tau$ and $\epsilon_{\sigma_1\sigma_2}^{(\alpha\alpha')}(v)\neq 0$. It is clear that
	\begin{equation}\label{eqn:cocycle_h}
	H(e_3)=H(e_1)\circ H(e_2).
	\end{equation}
	
	Now we turn on the gluing data $s$. Since all bundles that we constructed are equivariant with respective to the torus of the corresponding toric strata, there is an isomorphism $\cu{E}(\tau_2)\cong s_e^*\cu{E}(\tau_2)$, for any $e:\tau_1\to\tau_2$. Therefore, we can compose $h_{\tau_2\tau_1}$ with this isomorphism to get $h_{\tau_2\tau_1,s}$. But the cocycle condition (\ref{eqn:cocycle_g_s}) may not be satisfied due to the non-trivial gluing data $s$. Explicitly, by using (\ref{eqn:cocycle_h}), the composition
	$$h_{\tau_3\tau_1,s}^{-1}\circ( F_s(e_2)^*h_{\tau_2\tau_1,s})\circ h_{\tau_3\tau_2,s}$$
	is given by
	$$1_{\sigma}^{(\alpha)}(\tau_3)\mapsto s_{\tau_1\tau_2\tau_3}(\sigma^{(\alpha)})^{-1}1_{\sigma}^{(\alpha)}(\tau_3),$$
	where
	$$s_{\tau_1\tau_2\tau_3}(\sigma^{(\alpha)}):=s_{e_2}(m_{\tau_3}(\sigma^{(\alpha)}))s_{e_1}(m_{\tau_2}(\sigma^{(\alpha)}))s_{e_3}(m_{\tau_3}(\sigma^{(\alpha)}))^{-1}.$$
	Here $e_1:\tau_1\to\tau_2,e_2:\tau_2\to\tau_3,e_3=e_2\circ e_1$ and $m_{\tau_i}(\sigma^{(\alpha)})$ is the slope of $\varphi_{\tau_i'}$ on the cone $S_{\tau_i'}(\sigma^{(\alpha)})\in\Sigma_{\tau_i'}$. As $F_s(e_3)=F_s(e_1)\circ F_s(e_2)$, we have
	$$s_{\tau_1\tau_2\tau_3}(\sigma_1^{(\alpha)})=s_{\tau_1\tau_2\tau_3}(\sigma_2^{(\beta)}).$$
	Hence the quantity $$s_{\tau_1'\tau_2'\tau_3'}:=s_{\tau_1\tau_2\tau_3}(\sigma^{(\alpha)})$$
	only depends on the lifts $\tau_1'\subset\tau_2'\subset\tau_3'$ that $\sigma'$ contains. Also, if we choose other local representatives of $\varphi'$, then
	$$\hat{s}_{\tau_1'\tau_2'\tau_3'}=s_{\tau_1'\tau_2'\tau_3'}\left(s_{e_1}(a(\tau_3'))s_{e_2}(a(\tau_2'))s_{e_3}(a(\tau_3'))^{-1}\right),$$
	where $a(\bullet')$ is the slope some affine function. Hence, for any $\tau_1'\subset\tau_2'\subset\tau_3'$, we obtain an element $s_{\tau_1'\tau_2'\tau_3'}\in\bb{C}^{\times}$, so that $(s_{\tau_1'\tau_2'\tau_3'})_{\tau_1'\subset\tau_2'\subset\tau_3'}$ gives a $\bb{C}^{\times}$-valued C\v{e}ch 2-cocycle with respect to the simplicial structure on $L$ induced by $\msc{P}_{\pi}$. This defines a cohomology class
	$$o_{\bb{L}}(s):=[(s_{\tau_1'\tau_2'\tau_3'})_{\tau_1'\subset\tau_2'\subset\tau_3'}]\in H^2(\cu{W}',\bb{C}^{\times}).$$
	It is also clear that $o_{\bb{L}}(s)$ only depends on the cohomology class $[s]$ of $s$ in $H^1(\cu{W},\cu{Q}_{\msc{P}}\otimes\bb{C}^{\times})$. As a result, we obtain the \emph{obstruction map} as a group homomorphism
	$$o_{\bb{L}}:H^1(\cu{W},\cu{Q}_{\msc{P}}\otimes\bb{C}^{\times})\to H^2(\cu{W}',\bb{C}^{\times}).$$

	\begin{theorem}\label{thm:existence}
		If $o_{\bb{L}}([s])=1$, then the rank $r$ locally free sheaves $\{\cu{E}(\sigma)\}_{\sigma\in\msc{P}}$ can be glued to a locally free sheaf of rank $r$ over $X_0(B,\msc{P},s)$.
	\end{theorem}
	\begin{proof}
		If $o_{\bb{L}}([s])=1$, write
		$$s_{\tau_1'\tau_2'\tau_3'}=k_{\tau_1'\tau_2'}k_{\tau_3'\tau_3'}k_{\tau_1'\tau_3'}^{-1}.$$
		We modify $h_{\tau_1\tau_2,s}$ by
		$$\til{h}_{\tau_1\tau_2,s}:1_{\sigma}^{(\alpha)}(\tau_2)\mapsto\pm k_{\tau_1'\tau_2'}s_e(m_{\tau_2}(\sigma^{(\alpha)}))^{-1}F_s(e_1)^*1_{\sigma}^{(\alpha)}(\tau_1),$$
		where $\pm$ is determined by $h_{\tau_1\tau_2}$ and $\tau_1',\tau_2'$ are determined by the condition $\tau_1'\subset\tau_2'\subset\sigma^{(\alpha)}$. By using (\ref{eqn:cocycle_h}), one can easily check that
		$$\til{h}_{\tau_3\tau_1,s}^{-1}\circ(F_s(e_2)^*\til{h}_{\tau_2\tau_1,s})\circ\til{h}_{\tau_3\tau_2,s}=Id.$$
		Thus we can take the direct limit $\displaystyle{\lim_{\longrightarrow}}\,\cu{E}(\sigma)$ to glue $\{\cu{E}_{\sigma}\}_{\sigma\in\msc{P}}$ together.
	\end{proof}
	
	\begin{definition}
		If $o_{\bb{L}}([s])=1$, with a choice of $\{k_{\tau_1'\tau_2'}\}_{\tau_1'\subset\tau_2'}\subset\bb{C}^{\times}$ as in Theorem \ref{thm:existence}, we define $\cu{E}_0(\bb{L},{\bf{k}}_s):=\displaystyle{\lim_{\longrightarrow}}\,\cu{E}(\sigma)$.
	\end{definition}
	
	\begin{remark}
		As in the case of ample line bundles (see \cite{GS1}, paragraph right after the proof of Theorem 2.34), different choices of the constants $\{k_{\tau_1'\tau_2'}\}_{\tau_1'\subset\tau_2'}\subset\bb{C}^{\times}$ in the proof of Theorem \ref{thm:existence} may produce different locally free sheaves, and the notation $\cu{E}_0(\bb{L},{\bf{k}}_s)$ is to indicate the dependence.
	\end{remark}
	
	\begin{remark}
		The obstruction map $o_{\bb{L}}$ is a higher rank analog of the obstruction map defined in \cite[Theorem 2.34]{GS1} via open gluing data. Since the space of open gluing data is embedded into the space of closed gluing data (see \cite[Proposition 2.32]{GS1}), we can restrict $o_{\bb{L}}$ to the space of open gluing data.
	\end{remark}

	\section{Simplicity and smoothability}\label{sec:simple_smooth}
	
	As before, we assume that $\dim(B)=2$ and the polyhedral decomposition $\msc{P}$ is elementary (see Definition \ref{def:elementary}). Consider a tropical Lagrangian multi-section $\bb{L}$ of class $\msc{S}_{m,n}$ (see Definition \ref{def:proj_Lag}) and assume that the domain $L$ of $\bb{L}$ is connected. These imply that, at an unramified vertex $v'\in L$, $\varphi_{v'}$ can be represented by the piecewise linear function
	$$\varphi_k:=\begin{cases}
	0 & \text{ on }K_{\sigma_0},\\
	k\xi_1 & \text{ on }K_{\sigma_1},
	\\k\xi_2 & \text{ on }K_{\sigma_2},
	\end{cases}$$
	for $k\in\{m,n\}$, which corresponds to the line bundle $\cu{O}_{\bb{P}^2}(k)$ on $\bb{P}^2$. We also choose a gluing data $s$ such that $o_{\bb{L}}([s])=1$.\footnote{Recall that we always assume that all closed gluing data for both the fan and cone pictures are induced by open gluing data for the fan picture.}
	
	The main theorem of \cite{GS11} says that $X_0(B,\msc{P},s)$ can be smoothed out to give a toric degeneration $p:\cu{X}\to \Delta = \text{Spec}(\bb{C}[[t]])$.
	Now we would like to smooth out the pair $(X_0(B,\msc{P},s),\cu{E}_0(\bb{L},{\bf{k}}_s))$. To simplify notations, we write $X_0:=X_0(B,\msc{P},s)$ and $\cu{E}_0:=\cu{E}_0(\bb{L},{\bf{k}}_s)$. 
	Our strategy is to apply \cite[Corollary 4.6]{CM_pair}. So we need to compute the cohomology group $H^2(X_0,\cu{E}nd_0(\cu{E}_0))$. Higher cohomologies are usually hard to compute. But fortunately, we are in the 2-dimensional case and $X_0(B,\msc{P},s)$ is a Calabi-Yau, so it uffices to compute $H^0
	(X_0,\cu{E}nd_0(\cu{E}_0))$. We will handle the $r=2$ case and the $r\geq 3$ case separately.
	
	\subsection{Simplicity and smoothing in rank 2}\label{sec:rk2}
	
	For any vertex $\{v\}\in\msc{P}$, let
	\begin{align*}
	\Pi_v:H^0(X_0,\cu{E}nd(\cu{E}_0))\to & H^0(X_v,\cu{E}nd(\cu{E}(v))),
	\\\Pi_v^0:H^0(X_0,\cu{E}nd_0(\cu{E}_0))\to & H^0(X_v,\cu{E}nd_0(\cu{E}(v)))
	\end{align*}
	be the restriction maps. By Proposition \ref{prop:simple}, we know that $\Pi_v^0\equiv 0$ for $v\in S$. So, to prove simplicity of $\cu{E}_0(\bb{L},{\bf{k}}_s)$, it remains to study $\Pi_v^0$ for $v\notin S$.
	
	\begin{definition}\label{def:subgraph rank 2}
		Define
		\begin{equation}
		\begin{split}
		G(\msc{P}) & := \bigcup_{\tau\in\msc{P}^{(0)}\cup\msc{P}^{(1)}}\tau,\\
		G_0(\bb{L}) & := \bigcup_{\tau\in\msc{P}^{(0)}\cup\msc{P}^{(1)}:\tau\cap S=\emptyset}\tau.
		\end{split}
		\end{equation}
		They are embedded graphs in $B$ with $G_0(\bb{L})\subset G(\msc{P})$.
	\end{definition}

	Note that $V(G_0(\bb{L}))\amalg S=V(G(\msc{P})$. Hence if $G_0(\bb{L})=\emptyset$, we have $\Pi_v^0\equiv 0$ for all $v\notin S$. However, if $G_0(\bb{L})\neq\emptyset$, the group $H^0(X_v,\cu{E}nd_0(\cu{E}(v)))$ never vanishes for $v\notin S$ since $\cu{E}(v)\cong\cu{O}_{\bb{P}^2}(n)\oplus\cu{O}_{\bb{P}^2}(m)$.
	
	\begin{definition}\label{def:minimal cycle}
		A 1-cycle $\gamma\subset G_0(\bb{L})$ is called a \emph{minimal cycle} if there exists a 2-cell $\sigma\in\msc{P}$ such that $\partial\sigma=\gamma$.
	\end{definition}
	
	
	
	We now focus on the case $m=n+1$.
	\begin{definition}\label{def:simple}
		A tropical Lagrangian multi-section $\bb{L}\in\msc{S}_{n+1,n}$ is called \emph{simple} if $G_0(\bb{L})$ has no minimal cycles.
	\end{definition}
	
	For example, when $G_0(\bb{L})$ is a disjoint union of trees, $\bb{L}$ is always simple.
	
	
	\begin{theorem}\label{thm:simple}
		$\bb{L}\in\msc{S}_{n+1,n}$ is simple if and only if the locally free sheaf $\cu{E}_0(\bb{L},{\bf{k}}_s)$ is simple.
	\end{theorem}
	\begin{proof}
		Let $A$ be a non-zero section of $\cu{E}nd_0(\cu{E}_0(\bb{L},{\bf{k}}_s))$. Then there exists $v\in G_0(\bb{L})$ such that $\Pi_v^0(A)\neq 0$, $X_v \cong \mathbb{P}^2$ and $\cu{E}(v)\cong\cu{O}_{\bb{P}^2}(n) \oplus\cu{O}_{\bb{P}^2}(n+1)$. By reordering, we may assume $\Pi_v^0(A)$ is of the form
		$$\begin{pmatrix}
		\lambda_v & a_v\\
		0 & -\lambda_v
		\end{pmatrix},$$
		where $a_v$ is a section of $\cu{O}_{\bb{P}^2}(1)$ and $\lambda_v\in\bb{C}$. Since $L$ is connected, there exists a path in $G(\msc{P})$ connecting $v$ to some branch vertex $v'$. Since $\lambda_v$ is a constant, we must have $\lambda_v=0$ by simplicity of $\cu{E}(v')$ and continuity of the global section $A$. So we can write
		$$
		\Pi_v^0(A) = \begin{pmatrix}
		0 & a_v\\
		0 & 0
		\end{pmatrix},
		$$
		Since $a_v$ is a non-zero section of $\cu{O}_{\bb{P}^2}(1)$, there exists at least one vertex $\check{\sigma}\in\check{\msc{P}}$ such that $a_v$ is non-zero at the torus fixed point $X_{\sigma}\subset X_v$. Consider the minimal cycle $\gamma:=\partial\sigma$. We need to show that $\gamma$ lies inside $G_0(\bb{L})$. To see this, suppose there exists some vertex $v'\in V(\gamma)$ such that $\Pi_{v'}^0(A)=0$. Then by continuity, $a_v(A)$ must vanish at the torus fixed point $X_{\sigma}$ because $\check{v}$ and $\check{v}'$ share the common vertex $\check{\sigma}$ (see Figure \ref{fig:cycle6} below). Thus $\Pi_{v'}^0(A)\neq 0$ for all $v'\in V(\sigma)$, and hence $v'$ cannot be a branch vertex. In particular, $\gamma\subset G_0(\bb{L})$.
		\begin{figure}[ht]
			\includegraphics[width=40mm]{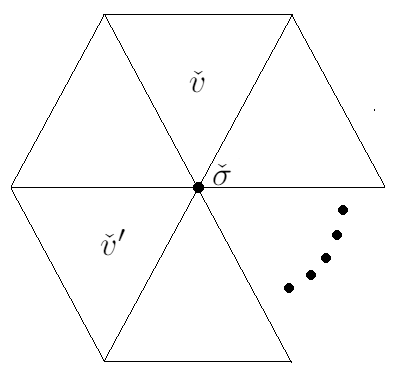}
			\caption{The dual of the minimal cycle $\gamma$.}
			\label{fig:cycle6}
		\end{figure}
		
		Conversely, suppose $G_0(\bb{L})$ has a minimal cycle $\gamma$. Let
		$$X_{\gamma}:=\bigcup_{v\in V(\gamma)} X_v$$
		and $\cu{E}_{\check{\gamma}}:=\cu{E}_0(\bb{L},{\bf{k}}_s)|_{ X_{\gamma}}$. We want to construct a non-zero section $A$ of $\cu{E}nd_0(\cu{E}_{\check{\gamma}})\to X_{\gamma}$ by gluing non-zero sections $\{s_{\check{v}}\}_{v\in V(\gamma)}$ of $\cu{O}(1)\to X_v$, $v\in V(\gamma)$, which has vanishing order 1 along the boundary divisors of $ X_{\gamma}$. To do this, we first define a cover $\{U_v\}_{v\in V(\gamma)}$ of $\gamma$ by
		$$U_v:=\{v\}\cup\bigcup_{\tau\in E(\gamma):v\in\tau}\text{Int}(\tau).$$
		Let $\{s_v\}_{v\in V(\gamma)}$ be sections of $\cu{O}(1)$ which vanish along the divisor correspond to the half-edges\footnote{Let $G$ be a graph and $H$ be a subgraph of $G$. An edge $\tau\in E(G)$ is called a \emph{half-edge} of $H$ if $\tau\notin E(H)$ and $\tau\cap H\neq\emptyset$.} of $\gamma$. Then for any $\tau\in E(\gamma)$ with vertices $e:\{v\}\to\tau,e':\{v'\}\to\tau$, $F_s(e)^*s_v$ and $F_s(e')^*s_{v'}$ can be regarded as non-zero holomorphic sections of $\cu{O}_{X_{\tau}}(1)\to X_{\tau}$. Thus
		$$\lambda_{vv'}:=\frac{F_s(e)^*s_v}{F_s(e')^*s_{v'}}$$
		is a meromorphic function on $ X_{\tau}$. But $\lambda_{vv'}$ has no zeros and poles, so $\lambda_{vv'}\in\bb{C}^{\times}$. Hence the data $\cu{L}:=(\{U_v\},\{\bb{C}\inner{s_v}\},\{\lambda_{vv'}\})$ defines a rank 1 local system on $\gamma\cong S^1$. Let $\sigma$ be such that $\partial\sigma=\gamma$. We extend $\cu{L}$ as follows: Cover $\sigma$ by
		$$U_v^{\sigma}:=U_v\cup \text{Int}(\sigma).$$
		If $U_v^{\sigma}\cap U_{v'}^{\sigma}\neq\emptyset$, then
		$$U_v^{\sigma}\cap U_{v'}^{\sigma}=\begin{cases}
		\text{Int}(\tau)\cup \text{Int}(\sigma) & \text{ if }U_v\cap U_{v'} = \text{Int}(\tau),
		\\\text{Int}(\sigma) & \text{ otherwise}.
		\end{cases}$$
		We define
		$$\til{\lambda}_{vv'}:=\begin{cases}
		\frac{F_s(e)^*s_v}{F_s(e')^*s_{v'}} \text{ if }U_v\cap U_{v'}=\text{Int}(\tau),
		\\\frac{F_s(f)^*s_v}{F_s(f')^*s_{v'}} \text{ otherwise},
		\end{cases},$$
		where $f:\{v\}\to\sigma,f':\{v'\}\to\sigma$. Clearly, $\til{\lambda}_{vv'}\in\bb{C}^{\times}$. If we let $h:\tau\to\sigma$, since $F_s(h)\circ F_s(e)=F_s(f)$, which is independent of $\tau$, the cocycle condition is satisfied. Thus we get a local system $\til{\cu{L}}\to\sigma$ extending $\cu{L}\to\gamma$. Since $\sigma$ is contractible, $\til{\cu{L}}$, and hence $\cu{L}$, are trivial as local systems. Therefore, we can find $\{c_v\}\subset\bb{C}^{\times}$ such that $\lambda_{vv'}=c_{v'}/c_v$. By definition, we have
		$$F_s(e)^*(c_vs_v)=F_s(e')^*(c_{v'}s_{v'})$$
		for all $\tau$ with vertices $v,v'$. Thus we obtain a section $A$ of $\cu{E}nd_0(\cu{E}_{\check{\gamma}})\to X_{\gamma}$ which vanishes along the boundary divisor of $ X_{\gamma}$. Extend $A$ by zero to the other toric components, we see that $\cu{E}nd_0(\cu{E}_0(\bb{L},{\bf{k}}_s))$ has a non-zero section.
	\end{proof}
	
	\begin{corollary}\label{cor:h2=0}
		$\bb{L}\in\msc{S}_{n+1,n}$ is simple if and only if $H^2(X_0(B,\msc{P},s),\cu{E}nd_0(\cu{E}_0(\bb{L},{\bf{k}}_s)))=0$ for any choice of ${\bf{k}}_s$.
	\end{corollary}
	\begin{proof}
		Since $ X_0(B,\msc{P},s)$ has Gorenstein singularities, Serre duality holds. The canonical sheaf is trivial by the Calabi-Yau condition.
	\end{proof}

	Because of Corollary \ref{cor:h2=0}, we can apply the smoothing result of \cite{CM_pair} to obtain the following
	
	\begin{corollary}\label{cor:smoothable}
		If $\bb{L}\in\msc{S}_{n+1,n}$ is simple, then the pair $( X_0(B,\msc{P},s),\cu{E}_0(\bb{L},{\bf{k}}_s))$ is smoothable for any choice of ${\bf{k}}_s$.
	\end{corollary}
	
	\begin{remark}
		The case when $m\geq n+2$ is actually much easier. By choosing sections of $\cu{O}_{\bb{P}^2}(m-n)$ which vanish only along boundary divisors, it is not hard to see that
		\begin{itemize}
			\item when $m=n+2$, $\cu{E}_0(\bb{L},{\bf{k}}_s)$ is simple if and only if $G_0(\bb{L})$ is a collection of vertices in $B\backslash S$;
			\item when $m \geq n+3$, $\cu{E}_0(\bb{L},{\bf{k}}_s)$ is simple if and only if $G_0(\bb{L})=\emptyset$.
		\end{itemize}
	\end{remark}
	
	
	
	It is not hard to construct such simple tropical Lagrangian multi-sections.
	First, to obtain a suitable pair $(B, \msc{P})$, we can start with an arbitrary pair $(\check{B},\check{\msc{P}}')$. By refining the polyhedral decomposition, we may assume that $(\check{B},\check{\msc{P}}')$ is already simple from the beginning. By further subdividing into elementary simplices and taking a discrete Legendre transform, we get an integral affine manifold $B$ with singularities equipped with a simple polyhedral decomposition $\msc{P}$.
	
	To construct a rank 2 tropical Lagrangian multi-section, we label the vertices in $\msc{P}$ so that
	\begin{enumerate}
		\item the total number of labelled vertices is even, and
		\item unlabelled vertices do not bound any minimal cycles.
	\end{enumerate}
	Then one can easily construct a closed topological surface $L$ and a branched covering $\pi:L\to B$ so that all the labelled vertices are branch points. The genus of $L$ is determined by the Riemann-Hurwitz formula:
	$$g(L)=\frac{\#\{\text{labelled vertices}\}}{2}-1$$
	Then we lift the polyhedral decomposition $\msc{P}$ to $L$, and try to glue our local model at each labelled vertex with $m=n+1$. Obviously, there are combinatorial obstructions for doing so, but the following condition is sufficient to guarantee that this can be done:
	\begin{itemize}
		\item [(E)] Every maximal cell has even number of branch vertices.
	\end{itemize}
	Under this condition, we obtain a simple tropical Lagrangian multi-section over $(B,\msc{P})$ of class $\msc{S}_{n,n+1}$. A couple of concrete examples are now in order.
	
	\begin{example}\label{eg:1}
		Let $\Xi$ be a simplex centered at origin and with affine length 5 on each edge and $\check{B}:=\partial\Xi$. Equip $\check{B}$ with the affine structure given in \cite[Example 1.18]{GS1}. This affine structure is integral because $\Xi$ is reflexive up to scaling. The proper faces of $\Xi$ give the standard polyhedral decomposition $\check{\msc{P}}_{std}$ on $\check{B}$. This decomposition is not simple but we can subdivide it into smaller triangles so that it becomes simple. More precisely, we may consider the subdivision shown in the following figure (where we have unfolded $\check{B}$).
		\begin{figure}[ht]
			\includegraphics[width=45mm]{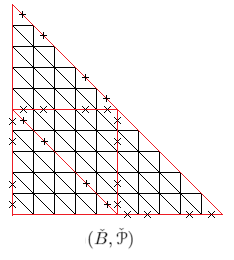}
			\caption{The red lines indicate the original standard polyhedral decomposition $\check{\msc{P}}_{std}$ and the black lines are its subdivision $\check{\msc{P}}$. The crosses denote the locations of the 24 affine singularities.}
			\label{fig:simple_simplex}
		\end{figure}
		Then we put the 24 affine singular points on distinct red edges that have affine length 1 so that the decomposition $\msc{P}$ is simple. Choose any strictly convex multi-valued piecewise linear function and let $(B,\msc{P})$ be the discrete Legendre transform of $(\check{B},\check{\msc{P}})$, which is also simple. Now we look at the following two configurations.
		\begin{figure}[ht]
			\includegraphics[width=100mm]{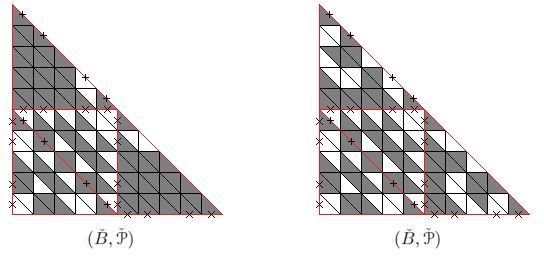}
			\caption{Shaded cells correspond to branch vertices on $(B,\msc{P})$.}
			\label{fig:simple_simplex}
		\end{figure}
		On the left (resp. right), we have 74 (resp. 58) branch points and there is a 2-fold branched covering map from a genus 36 (resp. 28) surface to $B$. Equip each ramification point with the local model defined before. As each maximal cell has even number of branched points in both configurations, we obtain two tropical Lagrangian multi-section $\bb{L}_1,\bb{L}_2\in\msc{S}_{m,n}$ over $(B,\msc{P})$. Clearly, there are no minimal cycles and therefore, when $m=n+1$, $\cu{E}_0(\bb{L}_i,{\bf{k}}_s)$ are simple for all ${\bf{k}}_s$ and $i=1,2$.
	\end{example}
	
	\begin{example}\label{eg:2}
		Here we consider the boundary $\check{B}$ of a cube centered at origin with affine length 2 on each edge. Equip it with the affine structure and polyhedral decomposition $\check{\msc{P}}_{std}$ introduced \cite[Example 1.18]{GS1}. Since $\check{B}$ is reflexive, this affine structure is integral. We subdivide $\check{\msc{P}}_{std}$ by elementary simplices as shown in Figure \ref{fig:simple_cube} (where, again, we have unfolded $\check{B}$ for a better visualization).
		This gives a simple polyhedral decomposition $\check{\msc{P}}$. Again, choose any strictly convex multi-valued piecewise linear function on $(\check{B},\check{\msc{P}})$ and define $(B,\msc{P})$ as the discrete Legendre transform of $(\check{B},\check{\msc{P}})$.
		\begin{figure}[ht]
			\includegraphics[width=50mm]{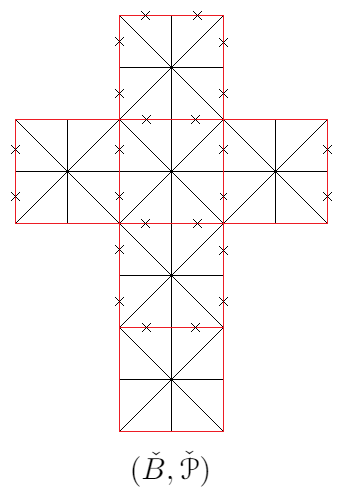}
			\caption{}
			\label{fig:simple_cube}
		\end{figure}
		
		Note that each maximal cell of $(B,\msc{P})$ has an even number of vertices. Hence, by taking a cover branched over all vertices, we obtain a tropical Lagrangian multi-section $\bb{L}\in\msc{S}_{m,n}$ of rank 2, whose domain is a genus 23 surface and where $m,n$ can be arbitrary distinct integers. Since all vertices are branch points, $\bb{L}$ must be simple.
		
		For a less obvious example, consider the configuration shown in Figure \ref{fig:O(1)}.
		\begin{figure}[H]
			\includegraphics[width=50mm]{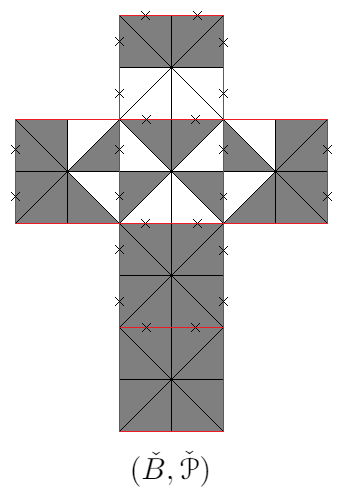}
			\caption{}
			\label{fig:O(1)}
		\end{figure}
		This produces a 2-fold branched covering map from a genus 17 surface to $B$, branched over the 36 vertices which correspond to the shaded maximal cells. By putting our local model at each ramification point, we get a tropical Lagrangian multi-section $\bb{L}\in\msc{S}_{m,n}$ without minimal cycles. Hence when $m=n+1$, the resulting locally free sheaf $\cu{E}_0(\bb{L},{\bf{k}}_s)$ is simple for all ${\bf{k}}_s$.
	\end{example}

	\subsection{Towards general rank}\label{sec:higher rank}
	
	As before, $B$ is a 2-dimensional integral affine manifold with singularities but now $\msc{P}$ can be \emph{any} polyhedral decomposition. Let $\bb{L}=(L,\msc{P}',\pi,\varphi)$ be a tropical Lagrangian multi-section of the following particular type:
	
	\begin{definition}\label{def:class_C}
		A tropical Lagrangian multi-section $\bb{L}$ of rank $r$ is said to be \emph{of class $\cu{C}$} if, for any $v\in S$, there exists a neighborhood $W_{v'}\subset L$ around $v$ such that $\pi|_{W_{v'}}:W_{v'}\to W_v$ is conjugate to the $r$-fold map $z\mapsto z^r$ on $\bb{C}$ and the slopes of the local representative $\varphi_{v'}$ satisfies
		\begin{enumerate}
			\item $m(\omega^{(\alpha)})\neq m(\omega^{(\beta)})$ for all $\omega\in\msc{P}$ and $\alpha\neq\beta$ and
			\item $m_{v'}(\sigma^{(\alpha)})-m_{v'}(\sigma^{(\beta)})\notin K_v(\sigma)^{\vee}\cap\cu{Q}_v^*$, for all $\sigma\in\msc{P}_{max}$ and $\alpha\neq\beta$.
		\end{enumerate}
	\end{definition}
	
	Suppose that Assumption \ref{assumption} holds, namely, for each branch vertex $v \in S$, there exists a rank $r$ toric vector bundle $\cu{E}(v)$ over $X_v$ whose cone complex satisfies the slope conditions in Definition \ref{def:class_C}. It can then be shown that the slope conditions in Definition \ref{def:class_C} imply that $\cu{E}(v)$ is simple \cite{Suen_preparation}. 
	Furthermore, a similar argument as in Section \ref{sec:E_0} gives an obstruction map $o_{\bb{L}}$ so that, when $o_{\bb{L}}([s])=1$ and $\bb{L}\in\cu{C}$, the collection of vector bundles $\{\cu{E}(v)\}_{v\in\msc{P}(0)}$ can be glued together. Since $\cu{E}(v)$ is simple for $v\in S$ and $\cu{E}(v)$ is a direct sum of line bundles for $v \not\in S$, we are in a situation very similar to the rank 2 case. So we can apply the same technique to study simplicity and smoothability.
	
	First of all, on $X_v$ where $v\notin S$, $\cu{E}(v)$ is given by a direct sum of line bundles
	$$\bigoplus_{\alpha=1}^r\cu{L}_{\varphi_{v^{(\alpha)}}},$$
	where $\varphi_{v^{(\alpha)}}$ is a local representative of $\varphi$ around the vertex $v^{(\alpha)}$. Thus, $\cu{E}nd(\cu{E}(v))$ is given by
	$$\bigoplus_{\alpha,\beta=1}^r\cu{L}_{\varphi_{v^{(\alpha)}}}^*\otimes\cu{L}_{\varphi_{v^{(\beta)}}}=\bigoplus_{\alpha,\beta=1}^r\cu{L}_{\varphi_{v^{(\beta)}}-\varphi_{v^{(\alpha)}}},$$
	So non-zero sections of $\cu{E}nd(\cu{E}(v))$ are basically sections of $\cu{L}_{\varphi_{v^{(\beta)}}-\varphi_{v^{(\alpha)}}}$.
	
	We now consider the fiber product
	$$P(\bb{L}):=L\times_BL.$$
	The polyhedral decomposition $\msc{P}'$ on $L$ induces a polyhedral decomposition on $P(\bb{L})$, namely,
	$$\til{\msc{P}}:=\{\sigma^{(\alpha)}\times_{\sigma}\sigma^{(\beta)}\,|\,\pi(\sigma^{(\alpha)})=\pi(\sigma^{(\beta)})=\sigma\}.$$
	Let $\til{G}(\bb{L})\subset P(\bb{L})$ be the graph given by the union of all edges in $\til{\msc{P}}$. 
	\begin{definition}\label{def:subgraph higher rank}
		We define the subgraph $\til{G}_0(\bb{L}) \subset \til{G}(\bb{L})$ by requiring:
		\begin{itemize}
			\item a vertex $(v^{(\alpha)},v^{(\beta)})$ of $\til{G}(\bb{L})$ is a vertex of $\til{G}_0(\bb{L})$ if and only if $v\notin S$ and the Newton polytope of $\varphi_{v^{(\beta)}}-\varphi_{v^{(\alpha)}}$ is non-empty, and
			\item an edge of $\til{G}(\bb{L})$ is an edge of $\til{G}_0(\bb{L})$ if and only if its vertices are in $\til{G}_0(\bb{L})$.
		\end{itemize}
	\end{definition}	
	For $\til{v}=(v^{(\alpha)},v^{(\beta)})\in \til{G}_0(\bb{L})$, we write $\varphi_{v^{(\beta)}}-\varphi_{v^{(\alpha)}}$ as $\varphi_{\til{v}}$. Also, as in Definition \ref{def:minimal cycle}, a 1-cycle $\til{\gamma}\subset \til{G}_0(\bb{L})$ is called a \emph{minimal cycle} if there exists a 2-cell $\til{\sigma}\in\til{\msc{P}}$ such that $\partial\til{\sigma}=\til{\gamma}$.
	
	\begin{example}\label{eg:3}
		Let us give a rank 3 example for which we have $\til{G}_0(\bb{L})=\emptyset$. Let $(B,\msc{P})$ be obtained as in Example \ref{eg:2}. Take a genus 22 closed surface $L$. There is a 3-fold covering map $\pi:L\to B$ branched over $B$ at all 48 vertices. At each branch point $v\in S$, we can use the local model that correspond to the toric vector bundle given in \cite[Example 4.2]{sections_toric_vb}. Namely, we consider the function $\varphi:|\Sigma'|\cong N_{\bb{R}}\to|\Sigma_{\bb{P}^2}|$, given by gluing
		\begin{align*}
		\varphi|_{\sigma_0^{(1)}}= & -\xi_1-2\xi_2, & \quad & \varphi|_{\sigma_1^{(1)}}= -3\xi_2, & \quad & \varphi|_{\sigma_2^{(1)}}=0,\\
		\varphi|_{\sigma_0^{(2)}}= & -2\xi_1, & \quad & \varphi|_{\sigma_1^{(2)}}=-\xi_1-\xi_2, & \quad & \varphi|_{\sigma_2^{(2)}}=-\xi_1+3\xi_2,\\
		\varphi|_{\sigma_0^{(3)}}=& -2\xi_1+3\xi_2, & \quad & \varphi|_{\sigma_1^{(3)}}=4\xi_1-3\xi_2, & \quad & \varphi|_{\sigma_2^{(3)}}=4\xi_1-2\xi_2,
		\end{align*}
		on each maximal cone $\sigma_i^{(\alpha)}\in\Sigma'(2)$. Each ramification point is of degree 3 by the Riemann-Hurwitz formula. Hence, $\pi$ locally looks like $z\mapsto z^3$ around each ramification point. In this case, it is also possible to put such local model at every vertex to obtain a tropical Lagrangian multi-section $\bb{L}$ of rank $3$ over $(B,\msc{P})$. Again, if $o_{\bb{L}}([s])=1$, we obtain a locally free sheaf $\cu{E}_0$ on $X_0$. Since $\cu{E}_0$ is simple on each maximal toric stratum, we have $\til{G}_0(\bb{L})=\emptyset$ and thus $\cu{E}_0$ is simple and the pair $(X_0, \cu{E}_0)$ is smoothable.
	\end{example}
	
	When $\til{G}_0(\bb{L})$ is non-empty, we have the following theorem, whose proof is similar to that of the ``only if'' part of Theorem \ref{thm:simple}:
	
	\begin{theorem}\label{thm:general_simple}
		Let $\bb{L}\in\cu{C}$ such that Assumption \ref{assumption} holds.
		Suppose that $\til{G}_0(\bb{L})$ has no minimal cycles, 
		and that the line bundle $\cu{L}(\til{v})$ associated to any $\til{v}\in\til{G}_0(\bb{L})$ admits a section $s_{\til{v}}\in H^0(X_v,\cu{L}(\til{v}))$ such that $s_{\til{v}}(p_{\til{v}})\neq 0$ for some torus fixed point $p_{\til{v}}\in X_v$.
		Then $\cu{E}_0(\bb{L},{\bf{k}}_s)$ is simple and hence the pair $(X_0(B,\msc{P},s),\cu{E}_0(\bb{L},{\bf{k}}_s))$ is smoothable for all choices of ${\bf{k}}_{s}$.
	\end{theorem}
	\begin{proof}
		Suppose $A\in H^0(X_0,\cu{E}nd_0(\cu{E}_0))$ is a non-zero section. By the same argument as in the proof of Theorem \ref{thm:simple}, but working on $P(\bb{L})$, we can assume for any vertex $\{v\}\in\msc{P}$, the diagonal terms of $\Pi_v^0(A)$ are all zero. Again, as $A$ is non-zero, there exists $v\in G_0(\bb{L})$ such that $\Pi_v^0(A)\neq 0$. Let $a_{\til{v}}$ be the $\cu{L}(\til{v})$-component of $\cu{E}nd_0(\cu{E}(v))$. By assumption, there exists $\til{v}\in\til{G}_0(\bb{L})$ and a torus fixed point $p_{\til{v}}\in X_v$ such that $a_{\til{v}}(p_{\til{v}})\neq 0$. The point $p_{\til{v}}$ corresponds to a maximal cell $\sigma\in\msc{P}$ that contains $v$. Lift $\sigma$ to the maximal cell $\til{\sigma}$ that contains $\til{v}$. All vertices of $\til{\sigma}$ are in $\til{G}_0(\bb{L})$, otherwise, $a_{\til{v}}(p_{\til{v}})=0$. It is now easy to see that $\partial\til{\sigma}\subset\til{G}_0(\bb{L})$ is a minimal cycle.
	\end{proof}
	
	\begin{remark}
		As $\til{G}_0(\bb{L})\to G_0(\bb{L})$ (here $G_0(\bb{L})$ is defined as in Definition \ref{def:subgraph rank 2}) is surjective and unramified, existence of minimal cycles in $\til{G}_0(\bb{L})$ is equivalent to existence of minimal cycles in $G_0(\bb{L})$. Thus we can weaken the assumption in Theorem \ref{thm:general_simple} to $G_0(\bb{L})$ having no minimal cycles.
	\end{remark}
	
	One advantage of the above theorem is that it works for \emph{any} polyhedral decompositions (not necessarily elementary ones). This provides a larger class of smoothable pairs. Finally, the converse of Theorem \ref{thm:general_simple} is also true if we know that, for any minimal cycle $\til{\gamma}\subset\til{G}_0(\bb{L})$ and vertex $\til{v}\in\til{\gamma}$, there exist sections of $\cu{L}(\til{v})$ that vanish along all the toric divisors which correspond to the half-edges of $\til{\gamma}$. The proof is similar to the rank $2$ case.

	\bibliographystyle{amsplain}
	\bibliography{geometry}

\providecommand{\bysame}{\leavevmode\hbox to3em{\hrulefill}\thinspace}
\providecommand{\MR}{\relax\ifhmode\unskip\space\fi MR }
\providecommand{\MRhref}[2]{%
  \href{http://www.ams.org/mathscinet-getitem?mr=#1}{#2}
}
\providecommand{\href}[2]{#2}
\begin{thebibliography}{10}

\bibitem{CLM_smoothing}
K.~Chan, N.~C. Leung, and Z.~N. Ma, \emph{Geometry of the {M}aurer-{C}artan
  equation near degenerate {C}alabi-{Y}au varieties}, J. Differential Geom., to
  appear, \href{https://arxiv.org/abs/1902.11174}{arXiv:1902.11174}.

\bibitem{CM_pair}
K.~Chan and Z.~N. Ma, \emph{Smoothing pairs over degenerate {C}alabi-{Y}au
  varieties}, Int. Math. Res. Not. IMRN, to appear,
  \href{https://arxiv.org/abs/1910.08256}{arXiv:1910.08256}.

\bibitem{CS_SYZ_imm_Lag}
K.~Chan and Y.-H. Suen, \emph{S{YZ} transforms for immersed {L}agrangian
  multisections}, Trans. Amer. Math. Soc. \textbf{372} (2019), no.~8,
  5747--5780.

\bibitem{sections_toric_vb}
S.~Di~Rocco, K.~Jabbusch, and G.~G. Smith, \emph{Toric vector bundles and
  parliaments of polytopes}, Trans. Amer. Math. Soc. \textbf{370} (2018),
  no.~11, 7715--7741.

\bibitem{FLTZ11a}
B.~Fang, C.-C.~M. Liu, D.~Treumann, and E.~Zaslow, \emph{A categorification of
  {M}orelli's theorem}, Invent. Math. \textbf{186} (2011), no.~1, 79--114.

\bibitem{FLTZ11c}
\bysame, \emph{The coherent-constructible correspondence and {F}ourier-{M}ukai
  transforms}, Acta Math. Sin. (Engl. Ser.) \textbf{27} (2011), no.~2,
  275--308.

\bibitem{FLTZ11b}
\bysame, \emph{The coherent-constructible correspondence and homological mirror
  symmetry for toric varieties}, Geometry and analysis. {N}o. 2, Adv. Lect.
  Math. (ALM), vol.~18, Int. Press, Somerville, MA, 2011, pp.~3--37.

\bibitem{FLTZ12}
\bysame, \emph{T-duality and homological mirror symmetry for toric varieties},
  Adv. Math. \textbf{229} (2012), no.~3, 1875--1911.

\bibitem{Ruddat_smoothing}
S.~Felten, M.~Filip, and H.~Ruddat, \emph{Smoothing toroidal crossing spaces},
  Forum Math. Pi \textbf{9} (2021), e7, 36 pp.

\bibitem{Friedman}
R.~Friedman, \emph{Global smoothings of varieties with normal crossings}, Ann.
  of Math. (2) \textbf{118} (1983), no.~1, 75--114.

\bibitem{Fukaya_asymptotic_analysis}
K.~Fukaya, \emph{Multivalued {M}orse theory, asymptotic analysis and mirror
  symmetry}, Graphs and patterns in mathematics and theoretical physics, Proc.
  Sympos. Pure Math., vol.~73, Amer. Math. Soc., Providence, RI, 2005,
  pp.~205--278.

\bibitem{gammage2021homological}
B.~Gammage and V.~Shende, \emph{Homological mirror symmetry at large volume},
  preprint (2021), \href{https://arxiv.org/abs/2104.11129}{arXiv:2104.11129}.

\bibitem{ganatra2020covariantly}
S.~Ganatra, J.~Pardon, and V.~Shende, \emph{Covariantly functorial wrapped
  {F}loer theory on {L}iouville sectors}, Publ. Math. Inst. Hautes \'{E}tudes
  Sci. \textbf{131} (2020), 73--200.

\bibitem{GHK11}
M.~Gross, P.~Hacking, and S.~Keel, \emph{Mirror symmetry for log {C}alabi-{Y}au
  surfaces {I}}, Publ. Math. Inst. Hautes \'{E}tudes Sci. \textbf{122} (2015),
  65--168.

\bibitem{GHS16}
M.~Gross, P.~Hacking, and B.~Siebert, \emph{Theta functions on varieties with
  effective anti-canonical class}, Mem. Amer. Math. Soc., to appear,
  \href{http://arxiv.org/abs/1601.07081}{arXiv:1601.07081}.

\bibitem{GS1}
M.~Gross and B.~Siebert, \emph{Mirror symmetry via logarithmic degeneration
  data. {I}}, J. Differential Geom. \textbf{72} (2006), no.~2, 169--338.

\bibitem{GS2}
\bysame, \emph{Mirror symmetry via logarithmic degeneration data, {II}}, J.
  Algebraic Geom. \textbf{19} (2010), no.~4, 679--780.

\bibitem{GS11}
\bysame, \emph{From real affine geometry to complex geometry}, Ann. of Math.
  (2) \textbf{174} (2011), no.~3, 1301--1428.

\bibitem{Gross-Siebert16}
\bysame, \emph{Theta functions and mirror symmetry}, Surveys in differential
  geometry 2016. {A}dvances in geometry and mathematical physics, Surv. Differ.
  Geom., vol.~21, Int. Press, Somerville, MA, 2016, pp.~95--138.

\bibitem{Kapranov-Schechtman16}
M.~Kapranov and V.~Schechtman, \emph{Perverse sheaves over real hyperplane
  arrangements}, Ann. of Math. (2) \textbf{183} (2016), no.~2, 619--679.

\bibitem{Kawamata-Namikawa}
Y.~Kawamata and Y.~Namikawa, \emph{Logarithmic deformations of normal crossing
  varieties and smoothing of degenerate {C}alabi-{Y}au varieties}, Invent.
  Math. \textbf{118} (1994), no.~3, 395--409.

\bibitem{HMS}
M.~Kontsevich, \emph{Homological algebra of mirror symmetry}, Proceedings of
  the {I}nternational {C}ongress of {M}athematicians, {V}ol.\ 1, 2 ({Z}\"urich,
  1994), Birkh\"auser, Basel, 1995, pp.~120--139.

\bibitem{KS_HMS_torus_fibration}
M.~Kontsevich and Y.~Soibelman, \emph{Homological mirror symmetry and torus
  fibrations}, Symplectic geometry and mirror symmetry ({S}eoul, 2000), World
  Sci. Publ., River Edge, NJ, 2001, pp.~203--263.

\bibitem{KS_scattering}
\bysame, \emph{Affine structures and non-{A}rchimedean analytic spaces}, The
  unity of mathematics, Progr. Math., vol. 244, Birkh\"auser Boston, Boston,
  MA, 2006, pp.~321--385.

\bibitem{Ma_thesis}
Z.~N. Ma, \emph{From {W}itten-{M}orse theory to mirror symmetry}, ProQuest LLC,
  Ann Arbor, MI, 2014, Thesis (Ph.D.)--The Chinese University of Hong Kong
  (Hong Kong).

\bibitem{branched_cover_fan}
S.~Payne, \emph{Toric vector bundles, branched covers of fans, and the
  resolution property}, J. Algebraic Geom. \textbf{18} (2009), no.~1, 1--36.

\bibitem{Ruddat_Siebert}
H.~Ruddat and B.~Siebert, \emph{Period integrals from wall structures via
  tropical cycles, canonical coordinates in mirror symmetry and analyticity of
  toric degenerations}, Publ. Math. Inst. Hautes \'{E}tudes Sci. \textbf{132}
  (2020), 1--82.

\bibitem{Schwarzenberger}
R.~L.~E. Schwarzenberger, \emph{Vector bundles on the projective plane}, Proc.
  London Math. Soc. (3) \textbf{11} (1961), 623--640.

\bibitem{Seidel-ICM}
P.~Seidel, \emph{Fukaya categories and deformations}, Proceedings of the
  {I}nternational {C}ongress of {M}athematicians, {V}ol. {II} ({B}eijing,
  2002), Higher Ed. Press, Beijing, 2002, pp.~351--360.

\bibitem{Seidel-K3}
\bysame, \emph{Homological mirror symmetry for the quartic surface}, Mem. Amer.
  Math. Soc. \textbf{236} (2015), no.~1116, vi+129.

\bibitem{SYZ}
A.~Strominger, S.-T. Yau, and E.~Zaslow, \emph{Mirror symmetry is
  {$T$}-duality}, Nuclear Phys. B \textbf{479} (1996), no.~1-2, 243--259.

\bibitem{Suen_preparation}
Y.-H. Suen, in preparation.

\bibitem{Suen_TP2}
\bysame, \emph{Reconstruction of holomorphic tangent bundle of complex
  projective plane via via tropical {L}agrangian multi-section}, New York J.
  Math. \textbf{27} (2021), no.~2, 1096--1114.

\bibitem{Tyurin}
A.~Tyurin, \emph{Geometric quantization and mirror symmetry}, preprint (1999),
  \href{https://arxiv.org/abs/math/9902027}{arXiv:9902027}.

\end{thebibliography}
\end{document}